\documentclass[reqno, a4paper, 10pt]{amsart}
\usepackage{amssymb,url, color, mathrsfs}
\usepackage[colorlinks=true, bookmarks=true, pdfstartview=FitH, pagebackref=true, linktocpage=true]{hyperref}
\usepackage[short,nodayofweek]{datetime}
\usepackage[pagewise,running,mathlines,displaymath, switch]{lineno}
\usepackage{times}
\usepackage{indentfirst, tabularx}
\usepackage[table]{xcolor}
\usepackage{float, hhline, colortbl}
\usepackage{graphicx}
\usepackage{arydshln} 
\usepackage{setspace}

\hypersetup{
pdftitle={Supercritical Sobolev} 
pdfauthor={ABC},
colorlinks = true,
linkcolor = magenta,
citecolor = blue,
}

\theoremstyle{plain}
\newtheorem{theorem}{Theorem}[section]

\newtheorem{lemma}[theorem]{Lemma}
\newtheorem{corollary}[theorem]{Corollary}

\theoremstyle{definition}

\theoremstyle{remark}

\def\R{\mathbb R}

\def\al{\alpha}
\def\om{\omega}
\def\Om{\Omega}
\def\be{\beta}
\def\ga{\gamma}
\def\de{\delta}
\def\De{\Delta} 

\def\vep{\varepsilon}
\def\na{\nabla}
\def\pa{\partial}
\def\lt{\left}
\def\rt{\right}

\def\vep{\varepsilon}

\def\H1r{H^1_{0,{\rm rad}}(B)}

\def\Hmr{H^m_{0,{\rm rad}}(B)}
\def\Cr{C^\infty_{0,{\rm rad}}(B)}
\def\CHR{C_{\mathsf{HR}}(n,m)}
\def\Snm{{\mathsf S}_{n,m}}
\def\Smn{\Sigma_{n,m}}
\def\Sb{{\mathsf S}_{\beta}}
\def\Un{{\mathsf U}_{n,\alpha}}
\def\Unm{{\mathsf U}_{n,m,\alpha}}
\def\Unma{{\mathsf U}_{n,m,\alpha_i}}
\def\s2{2^\star}
\def\2ms{2_m^\star}

\numberwithin{equation}{section}

\def\cfac#1{\ifmmode\setbox7\hbox{$\accent"5E#1$}\else\setbox7\hbox{\accent"5E#1}\penalty 10000\relax\fi\raise 1\ht7\hbox{\lower1.05ex\hbox to 1\wd7{\hss\accent"13\hss}}\penalty 10000\hskip-1\wd7\penalty 10000\box7 }

\title[A supercritical Sobolev type inequality in higher order Sobolev spaces]{A supercritical Sobolev type inequality in higher order Sobolev spaces and related higher order elliptic problems}

\author[Q.A. Ng\^o]{Qu\cfac{o}c Anh Ng\^o}

\address[Q.A. Ng\^o]{Department of Mathematics\\
College of Science, Vi\^{e}t Nam National University\\
H\`{a} N\^{o}i, Vi\^{e}t Nam.}
\email{\href{mailto: Q.A. Ng\^o <nqanh@vnu.edu.vn>}{nqanh@vnu.edu.vn}}
\email{\href{mailto: Q.A. Ng\^o <bookworm\_vn@yahoo.com>}{bookworm\_vn@yahoo.com}}

\author[V.H. Nguyen]{Van Hoang Nguyen}

\address[V.H. Nguyen]{Institute of Mathematics\\
Vietnam Academy of Science and Technology\\
Hanoi, Vietnam.}

\email{\href{mailto: V.H. Nguyen <vanhoang0610@yahoo.com>}{vanhoang0610@yahoo.com}}
\email{\href{mailto: V.H. Nguyen <nvhoang@math.ac.vn>}{nvhoang@math.ac.vn}}

\allowdisplaybreaks

\begin{document}

\begin{abstract}
A Sobolev type embedding for radially symmetric functions on the unit ball $B$ in $\mathbb R^n$, $n\geq 3$, into the variable exponent Lebesgue space $L_{2^\star + |x|^\alpha} (B)$, $2^\star = 2n/(n-2)$, $\alpha>0$, is known due to J.M. do \'O, B. Ruf, and P. Ubilla, namely, the inequality
\[
 \sup\Big\{\int_B |u(x)|^{2^\star+|x|^\alpha} dx   :  u\in \H1r,   \|\nabla u\|_{L^2(B)} =1\Big\} < +\infty
 \]
holds. In this work, we generalize the above inequality for higher order Sobolev spaces of radially symmetric functions on $B$, namely, the embedding
\[
\Hmr \hookrightarrow L_{2_m^\star + |x|^\alpha} (B)
\] 
with $2\leq m < n/2$, $2_m^* = 2n/(n-2m)$, and $\alpha>0$ holds. Questions concerning the sharp constant for the inequality including the existence of the optimal functions are also studied. To illustrate the finding, an application to a boundary value problem on balls driven by polyharmonic operators is presented. This is the first in a set of our works concerning functional inequalities in the supercritical regime. 
\end{abstract}

\date{\bf \today \; at   \currenttime}

\subjclass[2000]{46E35, 26D10, 35J30}

\keywords{supercritical Sobolev inequality; sharp constant; optimizer; higher order elliptic problems}

\maketitle


\section{Introduction}

The Sobolev embedding is a basic tool in many aspect of mathematical analysis. The classical one provides an optimal embedding from the Sobolev space $H^1 (\Omega)$ into the Lebesgue spaces $L_p(\Omega)$ with $p\leq 2^\star := 2n/(n-2)$, where $\Omega \subset \R^n$ with $n \geq 3$ is a bounded domain. If working in a larger class of ``rearrangement invariant'' Banach spaces rather than the class of $L_p$-spaces, the optimal exponent $2^\star$ can be slightly improved. For example, the following embedding is well-known
\[
H^1 (\Omega) \hookrightarrow L_{2^\star, 2} (\Omega),
\]
where $L_{2^\star, 2} (\Omega)$ is the well-known Lorentz space. In the literature, Sobolev embedding into non-rearrangement invariant spaces has recently captured attention. By choosing the variable exponent Lebesgue spaces $L_{p(x)}(B)$ as target spaces, where $B$ is the unit ball in $\R^n$, the authors in \cite{doO} are able to go beyond the critical threshold $2^\star$ when restricting to $H_{0, \rm rad}^1(B)$ the first order Sobolev space of radially symmetric functions about the origin. This special space is simply the completion of $\Cr$ under the norm
\[
\|u\|_{H^1_{0, \rm rad} (B)} = \Big( \int_B |\nabla u|^2 dx \Big) ^{1/2},
\]
where we denote by $\Cr$ the class of compactly supported, smooth, radially symmetric functions about the origin in $B$. The primary result in \cite{doO} states that given $\alpha > 0$ there exists a positive constant $\Un$ such that the supercritical Sobolev inequality
\begin{equation}\label{SobolevO}
\Un := \sup\Big\{\int_B |u(x)|^{\s2+|x|^\alpha} dx   :  u\in \H1r,   \|\nabla u\|_{L^2(B)} =1\Big\} < +\infty
\end{equation}
holds for any $u \in H_{0, \rm rad}^1(B)$. In other words, there is a continuous embedding
\[
H_{0, \rm rad}^1(B) \hookrightarrow L_{2^\star + |x|^\alpha} (B),
\]
where $L_{2^\star + |x|^\alpha}$ is the variable exponent Lebesgue space defined by
\[
L_{2^\star+ |x|^\alpha}(B) :=\Big\{u: B\to \R \quad \text{\rm is measurable}  :  \int_B |u(x)| ^{2^\star + |x|^\alpha} dx < + \infty\Big\}
\]
with norm
\[
\|u\|_{L_{2^\star + |x|^\alpha}(B)} = \inf\Big\{\lambda >0  :  \int_B \Big|\frac{u(x)}\lambda\Big|^{2^\star + |x|^\alpha} dx \leq 1\Big\}.
\]
As an application of \eqref{SobolevO}, which is quite a surprise, the authors are able to prove that the following elliptic equation
\begin{equation}\label{EllipticO}
\left\{
\begin{aligned}
-\Delta u &= u^{2^\star + |x|^\alpha - 1} & \text{ in }& B,\\
u &>  0 & \text{ in } & B,\\
u &=  0 & \text{ on }& \partial B,
\end{aligned}
\right.
\end{equation}
admits at least one solution. This result is somewhat intriguing because if one replace $|x|^\alpha$ by any non-negative constant, then \eqref{EllipticO} has no solution by the classical result of Pohozaev.

In this work, motivated by the supercritical Sobolev inequality \eqref{SobolevO}, first we generalize \eqref{SobolevO} for higher order Sobolev space of radially symmetric functions leading us to the following continous embedding
\[
\Hmr \hookrightarrow L_{\2ms + |x|^\alpha} (B),
\]
where the space $L_{\2ms + |x|^\alpha} (B)$ is precisely mentioned in Corollary \ref{varexspace} below. Then as an application of the inequality we present an existence result for solutions to the following polyharmonic equation
\begin{equation}\label{eq:supercriticaleq}
\left\{
\begin{aligned}
(-\De)^m u &=   u^{\2ms +|x|^\alpha -1} & \mbox{ in }& B,\\
u&>  0 & \mbox{ in } &B,\\
\pa^j_r u &=  0 & \mbox{ on } &\partial B, \quad j =0,\ldots,m-1 .
\end{aligned}
\right.
\end{equation}

To state our results, several notations and conventions are needed. First, for an integer $m \geq 1$, we denote
\[
\nabla^m = \begin{cases}
\De^{m/2} &\mbox{if $m$ is even},\\
\na \De^{(m-1)/2} &\mbox{if $m$ is odd}.
\end{cases}
\]
By $H^{m}_0(B)$ we mean the usual Sobolev space on $B$, which is the completion of $C_0^\infty(B)$ under the norm
\[
\|u\|_{H^{m}_0(B)} = \Big( \int_B |\nabla^ mu|^2 dx \Big) ^{1/2}.
\] 
Then analogue to $H_{0, \rm rad}^1(B)$, we denote by $\Hmr$ the completion of $\Cr$ with respect to the preceding norm. Given $\alpha>0$, we are interested in whether or not the following inequality
\begin{equation}\label{SobolevSuper}
\Unm := \sup\Big\{\int_B |u(x)|^{\2ms + |x|^\alpha} dx   :  u\in \Hmr,  \|\nabla^m u\|_{L^2(B)} \leq 1\Big\} < + \infty.
\end{equation}
holds for some constant $\Unm > 0$. Note that, in \eqref{SobolevSuper} and under the condition $n > 2m$, the number $\2ms = 2n/(n-2m)$ is also the critical exponent for the following Sobolev inequality with the sharp constant $\Snm$
\begin{equation}\label{SobolevInequality}
\|u\|_{L^{\2ms}(B)} \leq \Snm \|\nabla^m u\|_{L^2(B)}
\end{equation}
for any $u \in H^{m}_0(B)$. It is well-known that the sharp constant $\Snm$ can be characterized by
\[
\Snm = \sup\big\{\|u\|_{L^{\2ms}(B)}   :  u\in \Hmr,   \|\nabla^m u\|_{L^2(B)} =1\big\}
\]
(see the formulas \eqref{eq:tiemcan1} and \eqref{eq:tiemcan2} below) and if we let
\begin{equation*}
\Smn := \sup\Big\{\int_B |u(x)|^{\2ms} dx   :  u\in \Hmr,   \|\nabla u\|_{L^2(B)} =1\Big\} ,
\end{equation*}
then we immediately have
\[
\Smn = \Snm^{\2ms}.
\]
The first main result in this paper answers the above question affirmatively.

\begin{theorem}\label{Supercritical}
Let $1 \leq m < n/2$ and $\alpha >0$. Then
\begin{equation*}\label{eq:SuperSobolev}
\sup\Big\{\int_B |u(x)|^{\2ms + |x|^\alpha} dx   :  u\in \Hmr,  \|\nabla^m u\|_{L^2(B)} \leq 1\Big\} < + \infty.
\end{equation*}
\end{theorem}

Apparently, the case $m=1$ was already studied in \cite{doO}, our contribution is for the case $m \geq 2$. Clearly, a consequence of Theorem \ref{Supercritical} is that the space $\Hmr$ can be continuously embedded into the variable exponent Lebesgue space $L_{\2ms + |x|^\alpha} (B)$ mentioned earlier. An exact statement of this fact is as follows:

\begin{corollary}\label{varexspace}
Let $1 \leq m < n/2$ and $\alpha >0$. Then the following embedding is continuous
\[
\Hmr \hookrightarrow L_{\2ms + |x|^\alpha}(B),
\]
where $L_{\2ms + |x|^\alpha}$ is the variable exponent Lebesgue space defined by
\[
L_{\2ms + |x|^\alpha}(B) :=\Big\{u: B\to \R \quad \text{\rm is measurable}  :  \int_B |u(x)| ^{\2ms + |x|^\alpha} dx < + \infty\Big\}
\]
with norm
\[
\|u\|_{L_{\2ms + |x|^\alpha}(B)} = \inf\Big\{\lambda >0  :  \int_B \Big|\frac{u(x)}\lambda\Big|^{\2ms + |x|^\alpha} dx \leq 1\Big\}.
\]
\end{corollary}

In view of Theorem \ref{Supercritical}, there exists a sharp constant $\Unm > 0$ as already given in \eqref{SobolevSuper}. In this sense, it is natural to ask whether or not the sharp constant $\Unm$ is attained. To obtain the attainability of the sharp constant $\Unm$ and inspired by \cite[Theorem 1.3]{doO}, we first establish certain estimates between $\Unm$ and $\Smn$ as shown in the following.

\begin{theorem}\label{Strict}
Let $1 \leq m < n/2$ and $\alpha >0$. Then, there always holds
\begin{equation}\label{U>=S}
\Unm \geq \Smn
\end{equation}
Moreover, if 
\begin{equation}\label{eq:alphacond}
0 < \alpha \leq n- 2m,
\end{equation}
then there holds
\begin{equation}\label{U>S}
\Unm > \Smn .
\end{equation}
Finally, the following limit
\begin{equation}\label{U-infinity}
\lim_{\alpha \nearrow +\infty} \Unm = \Smn
\end{equation}
occurs.
\end{theorem}

In view of \eqref{U>=S}, it is now clear to see how reasonable the condition ${\mathsf U}_n > \Sigma_n$ appearing in \cite[Theorem 1.4]{doO} is. Compared to \cite[Theorem 1.3]{doO}, it is clear that, even when $m=1$, which was also studied in \cite{doO}, the range for $\alpha$ in \eqref{eq:alphacond} is significant improved.

Then the following result provides us a criteria in which the sharp constant $\Unm$ is attained.

\begin{theorem}\label{Attain}
Let $1 \leq m < n/2$ and $\alpha >0$. If
\[
\Unm > \Smn,
\] 
then the sharp constant $\Unm$ is attained.
\end{theorem}

Combining Theorems \ref{Strict} and \ref{Attain} we deduce that the sharp constant $\Unm$ is attained if $0 < \alpha \leq n-2m$ and it is likely that the sharp constant $ \Smn$ serves as a threshold for the existence of optimizers for $\Unm$. Although we cannot say any about the inequality \eqref{U>S} whenever $\alpha > n-2m$, the limit in \eqref{U-infinity} might lead us to a non-existence of optimizers for $\Unm$ when $\alpha$ is very large. If this is not the case, we expect to see certain monotonicity of $\Unm$ with respect to $\alpha$; see \cite{VHN-PROCA} for related results. We take this chance to mention that in the literature a similar phenomenon appears in  the Adimurthi--Druet inequality, an improvement of the standard Moser--Trudinger inequality by adding a $L^2$-type perturbation; see \cite{MT2019}.

Finally, we study the existence of solutions to \eqref{eq:supercriticaleq}. Our existence result reads as follows.

\begin{theorem}\label{Superequation}
Let $1 \leq m < n/2$ and $0 < \alpha \leq n-2m$. Then there exists at least one weak solution to \eqref{eq:supercriticaleq}.
\end{theorem}

To look for a solution to \eqref{eq:supercriticaleq}, we employ variational techniques. In this way, a solution to \eqref{eq:supercriticaleq} is found as a critical point of the associated Euler--Lagrange energy functional defined on $\Hmr$. In turn, such a solution is radially symmetric. Taking the recent work \cite{CLL18} into account, we expect to see more solution to \eqref{eq:supercriticaleq} instead of the radial ones.

This paper is organized as follows:

\tableofcontents


This is the first paper in a set of our works concerning functional inequalities in the supercritical regime. In the next paper \cite{NN19}, we shall address supercritical Moser--Trudinger inequalities.

\section{Preliminaries}

This section is to prepare some auxiliary results which will be used in the proof of the main Theorems. 

\subsection{Useful inequalities}

This subsection is devoted to useful inequalities. For clarity, let us denote the following constants. Given $a<n-4$, we let
\[
c_{n,a} = \Big( \frac{(n+a)(n-a -4)}4\Big)^2
\]
and let
\[
\CHR =
\left\{
\begin{aligned}
& \Big(\frac{n+2m-4}2 \Big)^2 \prod_{i=0}^{\lfloor m/2 \rfloor-2} c_{n,4i} &&\mbox{ if }m=2k,\\
& \Big(\frac{n-2}2 \Big)^2 \Big(\frac{n+2m-4}2 \Big)^2 \prod_{i=0}^{\lfloor m/2 \rfloor -2} c_{n,2+4i}  &&\mbox{ if } m = 2k+1,
\end{aligned}
\right.
\]
or more precisely,
\[
\CHR =
\left\{
\begin{aligned}
& \frac4{(n-4k)^2} \prod_{i=0}^{k-1} \frac{(n+4i)^2(n-4i-4)^2}{16} &&\mbox{ if }m=2k,\\
& \frac{(n+4k-2)^2}{(n-2)^2}\prod_{i=0}^{k-1} \frac{(n-2 + 4i)^2(n-2-4i)^2}{16} &&\mbox{ if } m = 2k+1.
\end{aligned}
\right.
\]
The first result is a sharp Hardy--Rellich inequality in $\Hmr$. Although our inequality is in the sharp form, technically, we do not really require such a form in our analysis.

\begin{lemma}[Hardy--Rellich inequality in $\Hmr$]\label{HRhigher}
 Let $2\leq m < n/2$. Then there holds
\begin{equation}\label{eq:HR}
\|\nabla^m u\|_{L^2(B)} \geq \CHR \int_B \frac{|\na u|^2}{|x|^{2(m-1)}} dx
\end{equation}
for any $u\in \Hmr$.
\end{lemma}

\begin{proof}
By density, it suffices to prove that \eqref{eq:HR} holds for any $u \in \Cr$. First, we recall the following well-known weighted Hardy inequality
\begin{equation}\label{eq:Hardy}
 \int_B \frac{|\nabla u|^2}{|x|^a} dx 
 \geq \Big(\frac{n-2-a}2\Big)^2 \int_B \frac{u^2}{|x|^{a+2}} dx 
\end{equation}
for any $u \in C_0^\infty(B)$ and with $0\leq a <n-2$. 
We next provide a similar Rellich inequality that connects first to second order derivatives of radial functions, namely, for $0 \leq a < n-2 $, we shall prove
\begin{align}\label{eq:Rtype}
\int_B \frac{ (\Delta u)^2}{ |x|^a} dx  \geq \frac{(n+a)^2}4 \int_B \frac{|\na u|^2}{|x|^{a+2}} dx
\end{align}
holds for any $u \in \Cr$. Indeed, let $u \in \Cr$ and observe that $\Delta u = u''(r) + ((n-1)/r) u'(r)$. From this we obtain
\begin{align*}
\int_B \frac{ (\Delta u)^2}{ |x|^a} dx =& \omega_{n-1} \int_0^1 \Big( u''(r) + \frac{n-1} r u'(r) \Big) ^2 r^{n-a -1} dr\notag\\
 =&\omega_{n-1} \int_0^1 (u''(r))^2 r^{n-a-1} dr + (n-1) \omega_{n-1} \int_0^1 [(u'(r))^2]' r^{n-a -2} dr \notag\\
& + (n-1)^2 \omega_{n-1} \int_0^1 (u'(r))^2 r^{n-a-3} dr\notag\\
=& \omega_{n-1} \int_0^1 (u''(r))^2 r^{n-a-1} dr \notag \\
&+ (n-1)(a+1)\omega_{n-1} \int_0^1 (u'(r))^2 r^{n-a-3} dr\notag\\
 \geq &\frac{(n+a)^2}4 \omega_{n-1} \int_0^1 (u'(r))^2 r^{n-a-3} dr\notag\\
 =&\frac{(n+a)^2}4 \int_B \frac{|\na u|^2}{|x|^{a+2}} dx,
\end{align*}
where the third equality comes from integration by parts while the inequality comes from \eqref{eq:Hardy}. We are now in position to conclude the lemma. There are two possible cases:

\noindent\textbf{Case 1}. Suppose $m = 2k$ with $k \geq 1$. Using \cite[Theorem $1.9$(i)]{TZ} we get
\begin{align*}
\int_B (\Delta^k u)^2 dx \geq \Big(\prod_{i=0}^{k-2} c_{n,4i} \Big) \int_B \frac{(\Delta u)^2 }{|x|^{4(k-1)}} dx
\end{align*}
for any $u \in \Cr$. Combining the previous inequality with \eqref{eq:Rtype}, namely,
\[
\int_B \frac{(\Delta u)^2 }{|x|^{4(k-1)}} dx \geq \Big(\frac{n+4k-4}2 \Big)^2 \int_B \frac{|\nabla u|^2 }{|x|^{4k-2}} dx
\]
we arrive at \eqref{eq:HR} for $m$ even and for all $u \in \Cr$.

\noindent\textbf{Case 2}. Suppose $m =2k+1$ for some $k\geq 1$. By using \cite[Theorem $1.9$(ii)]{TZ}, we get
\begin{align*}
\int_B |\nabla \Delta^k u|^2 dx & \Big(\frac{n-2}2 \Big)^2 \Big(\prod_{i=0}^{k-2} c_{n,2+4i} \Big) \int_B \frac{ (\Delta u)^2 }{|x|^{ 4k-2}} dx.
\end{align*}
Finally, we make use of \eqref{eq:Rtype}, namely,
\[
\int_B \frac{(\Delta u)^2 }{|x|^{4k-2}} dx \geq \Big(\frac{n+4k-2}2 \Big)^2 \int_B \frac{|\nabla u|^2 }{|x|^{4k }} dx
\] 
to get the desired inequality.
\end{proof}

Before going futher, it is worth noting that without restricting to functions in $\Cr$, the sharp form of the inequality \eqref{eq:Rtype} for functions in $C_0^\infty(B)$ has already known and the sharp constant for \eqref{eq:Rtype} is technically complicated; see \cite[Theorem 1.7]{TZ}. To be more precise, the sharp constant computed by authors in \cite{TZ} is given by
\[
\min_{k = 0, 1, ...}\frac{\frac 1{16}(n+a)^2(n-4-a)^2+k(n-2+k)}{\frac 14 ( n-4-a)^2+k(n-2+k)},
\]
which cannot be bigger than $(n+a)^2/4$. As clearly demonstrated in \cite{TZ}, the sharp constant equals $(n+a)^2/4$ only if $a$ is closed to zero. However, in the case of radially symmetric functions, the sharp constant is always $(n+a)^2/4$ by testing the functions
\[
u_\epsilon(x) = (1-\varphi_\epsilon(x)) \varphi(x)|x|^{-(n-a-4)/2},
\]
where $\varphi$ is a cut-off function in $C_0^\infty(B)$ such that $0\leq \varphi \leq 1$, $\varphi(x) = 1$ if $|x| \leq 1/2$, and $\varphi_\epsilon(x) =\varphi(x/\epsilon)$. Similarly, the constant $\CHR$ in \eqref{eq:HR} is sharp. This can be verified by using the test functions 
\[
u_\epsilon(x) = (1-\varphi_\epsilon(x))\varphi(x)|x|^{-(n-2m)/2}.
\] 
 
Our next result is the sharp Sobolev inequality with fractional dimension $\beta$; see \cite{VHN} and references therein.

\begin{lemma}\label{SobolevFractional}
Let $\beta >2$. There exists a positive constant $\Sb>0$ such that
\begin{equation}\label{eq:fracdimSobolev}
\Sb \int_0^1 |u'(r)|^2 r^{\beta -1} dr \geq \Big( \int_0^1 |u(r)|^{\frac{2\beta}{\beta -2}} r^{\beta -1} dr \Big) ^{\frac{\beta -2}\beta}
\end{equation}
for $u \in C_0^\infty([0,1))$.
\end{lemma} 

Making use of Lemma \ref{SobolevFractional}, we can prove a generalization of the supercritical Sobolev inequality \eqref{SobolevO} of do \'O, Ruf, and Ubilla to the fractional dimension $\beta$. 

\begin{lemma}\label{fract}
Let $\beta >2$ and $q(r) = \frac{2\beta}{\beta -2} + f(r)$ with $f:[0,1) \to [0,\infty)$ so that $f$ is continuous and satisfies the following conditions:
\begin{enumerate}
\item[($f_1$)] $f(0) =0$ and $f(r)>0$ for $r >0$;
\item[($f_2$)] there exists some $c > 0$ such that
\[
f(r) \leq \frac c{-\ln r}
\] 
for $r$ near $0$.
\end{enumerate}
Then, for all $a >0$,
\[
S_\beta(f,a):=\sup_{u \in \mathscr A_a}  \int_0^1 |u(r)|^{q(r)} r^{\beta -1} dr 
\]
is finite where
\[
\mathscr A_a = \Big\{ u\in C_0^\infty([0,1)),  \int_0^1 |u'(r)|^2 r^{\beta -1} dr \leq a \Big\}.
\]
\end{lemma}

\begin{proof}
Let $u \in \mathscr A_a$ be arbitrary, namely, $u \in C_0^\infty([0,1))$ and $\int_0^1 |u'(r)|^2 r^{\beta -1} dr \leq a$. Our aim is to estimate $\int_0^1 |u(r)|^{q(r)} r^{\beta -1} dr$. By H\"older's inequality, we estimate
\begin{align*}
|u(r)| &= \Big|-\int_r^1 u'(s) ds\Big| \\
&\leq \Big( \int_r^1 |u'(s)|^2 s^{\beta -1} ds \Big) ^{1/2} \Big( \int_r^1 s^{1-\beta} ds \Big) ^{1/2}\\
&\leq \Big( a \frac{r^{2-\beta} -1}{\beta -2} \Big) ^{1/2}.
\end{align*} 
In particular, because $r \in (0,1)$ and $\beta>2$, we then get
\begin{equation}\label{eq:est1}
|u(r)| \leq a^{1/2} \frac{r^{(2-\beta)/2}}{(\beta -2)^{1/2}} 
\end{equation}
for any $r \in (0,1)$. Now taking $r_0$ in such a way that $ a (r_0^{2-\beta} -1)/(\beta -2) =1$, namely, $r_0 = (a/(a+\beta -2)) ^{1/(\beta -2)}$, we then have
\begin{equation}\label{eq:est2}
|u(r)| \leq 1
\end{equation}
for any $r \in (r_0,1)$. We now have by \eqref{eq:est2} the following
\begin{equation}\label{eq:bound1}
\int_{r_0}^1 |u(r)|^{q(r)}r^{\beta-1} dr \leq \int_{r_0}^1 dr \leq 1.
\end{equation}
We next estimate the integral on $(0,r_0)$. By \eqref{eq:est1} we know that
\begin{align*}
\int_{0}^{r_0} |u(r)|^{q(r)}r^{\beta-1} dr &= \int_0^{r_0} |u(r)|^{\frac{2\beta}{\beta-2}}|u(r)|^{f(r)}r^{\beta -1} dr\\
&\leq \int_0^{r_0} |u(r)|^{\frac{2\beta}{\beta-2}} \Big( \frac{a r^{2-\beta}}{\beta -2} \Big) ^{f(r)/2 } r^{\beta-1} dr.
\end{align*}
Let 
\[ 
g(r) = \Big( \frac{a r^{2-\beta}}{\beta -2} \Big) ^{f(r)/2 },\quad r\in (0,r_0].
\]
Note that $g(r)$ is continuous on $(0,r_0]$ and
\[
\ln g(r) = \frac{f(r)}2 \Big( \ln \frac{a}{\beta -2} -(\beta-2)\ln r \Big) ,
\]
which yields
\[
\limsup_{r\to 0} \ln g(r) \leq c \frac{\beta -2}2
\]
by the assumptions ($f_1$)--($f_2$) on $f$. Consequently, the function $g$ is bounded on $(0,r_0]$. Putting 
\[
C_0 = \sup_{t\in (0,r_0]} g(t),
\] 
then we have
\begin{equation}\label{eq:bound2}
\int_0^{r_0} |u(r)|^{q(r)}r^{\beta-1} dr \leq C_0 \int_0^{r_0} |u(r)|^{\frac{2\beta}{\beta-2}} r^{\beta-1} dr \leq C_0 (\Sb a)^{\frac{\beta}{\beta -2}},
\end{equation}
where $\Sb$ is the sharp constant in the Sobolev inequality \eqref{eq:fracdimSobolev}. Combining \eqref{eq:bound1} and \eqref{eq:bound2} proves the lemma.
\end{proof}

Compared to \cite[Section 2]{doO}, our condition ($f_1$) is exactly the same as that of \cite{doO}, however, our condition ($f_2$) for $r$ near $0$ is weaker than that of \cite{doO}, and the most important is that we do not assume any behavior of $f$ near $1$ as indicated in ($f_3$) of \cite{doO}.


\subsection{Estimates for bubbles}

We now denote by $u_1^*$ the following bubble-shaped function
\[
u_1^*(x) = \Big( \frac 2{1 + |x|^2} \Big)^{\frac{n-2m}2}.
\]
Clearly, $u_1^*$ solves the following equation
\begin{equation*}
(-\De)^m u = u ^{\2ms -1}
\end{equation*}
in $\R^n$; see \cite{WeiXu99}. Then, for each $ \varepsilon >0$, we scale $u_1^*$ to obtain $u_\varepsilon^*$ in the following way
\[
u_\varepsilon^*(x) = \varepsilon^{-\frac{n-2m}2} u_1^*(\varepsilon^{-1} x) ,
\]
namely,
\[
u_\varepsilon^*(x) = \Big( \frac{2\varepsilon}{\varepsilon^2 + |x|^2 } \Big)^{\frac{n-2m}2}.
\]
Clearly, $u_\varepsilon^*$ also solves the above equation, namely,
\begin{equation*}
(-\De)^m u_\varepsilon^* = (u_\varepsilon^*)^{\2ms -1}
\end{equation*}
in $\R^n$. It is well-known that all functions $u_\varepsilon^*$ are the optimizers of the Sobolev inequality \eqref{SobolevInequality}, namely,
\[
 \Snm^2\int_{\R^n} |\nabla^m u_\varepsilon^*|^2 dx = \Big( \int_{\R^n} (u_\varepsilon^*)^{\2ms} dx \Big)^{2/\2ms}.
\]
In this sense, there holds
\[
\int_{\R^n} |\nabla^m u_\varepsilon^*|^2 dx = \int_{\R^n} (u_\varepsilon^*)^{\2ms} dx = \Snm^{-n/m}.
\]
Let $\eta$ be a cut-off function on $B$, which is of class $C_0^\infty(B)$ and radially symmetric. In addition, we require $0\leq \eta \leq 1$ everywhere and $\eta(x) =1$ if $|x| \leq 1/2$. When $\varepsilon$ is near zero, the following expansions for $\eta u_\varepsilon^*$ are well-known
\begin{equation}\label{eq:tiemcan1}
\int_B | \nabla^m(\eta u_\varepsilon^*) |^2 dx = \Snm^{-n/m} + O(\varepsilon^{n-2m})_{\varepsilon \searrow 0}
\end{equation}
and
\begin{equation}\label{eq:tiemcan2}
\int_B (\eta u_\varepsilon^*)^{\2ms} dx = \Snm^{-n/m} + O(\varepsilon^{n})_{\varepsilon \searrow 0};
\end{equation}
see \cite[proof of Theorem $7.23$]{GGS}.

Our first set of results in this subsection concerns the expansion of $\int_B |\eta u_\varepsilon^*|^{\2ms + |x|^\alpha} dx$ in terms of $\varepsilon$. 

\begin{lemma}\label{Tiemcan}
Let $\alpha , C> 0$ and denote
\[
v_\varepsilon (x)= C \eta (x) u_\varepsilon^* (x).
\] 
There exists a constant $\mathscr C_1>0$ such that for $\varepsilon >0$ small enough
\begin{equation}\label{eq:tiemcan3}
\int_B |v_\varepsilon|^{\2ms + |x|^\alpha} dx 
= 
\left\{
\begin{aligned}
&C^{\2ms}\Snm^{-n/m} + C^{\2ms} \mathscr C_1 |\ln \varepsilon| \varepsilon^\alpha + o(\varepsilon^\alpha|\ln \vep|)_{\varepsilon \searrow 0} &\mbox{ if }\alpha < n,\\
&C^{\2ms}\Snm^{-n/m} + O(\varepsilon^{n(1-\gamma)})_{\varepsilon \searrow 0}&\mbox{ if } \alpha \geq n,
\end{aligned}
\right.
\end{equation}
for arbitrary $0< \gamma < 1/2$ but fixed.
\end{lemma}

\begin{proof}
It follows from the definition of $u_\varepsilon^*$ that $C u_\varepsilon^*(x) \leq 1$ if and only if
\[
|x| \geq \sqrt{ A_{n,m}^{ 2/(n-2m)} \varepsilon -\varepsilon^2} =:a_\varepsilon,
\]
where
\[
A_{n,m} = 2^\frac{n-2m}{2} C.
\]
Notice that $a_\varepsilon \sim \sqrt{\vep}$ as $\vep\to 0$. For any $0 < \gamma <  1/2$ but fixed, we then have $a_\varepsilon < \varepsilon^\gamma< 1/2 $ provided $\varepsilon >0$ is small enough. Hence, $v_\varepsilon \leq 1$ on $B\setminus B_{\varepsilon^\gamma}$ which then implies
\begin{align*}
0 \leq\int_{B\setminus B_{\varepsilon^\ga}} |v_\varepsilon|^{\2ms + |x|^\alpha} dx &\leq \int_{B\setminus B_{\varepsilon^\ga}} |v_\varepsilon|^{\2ms} dx\\
&\leq C^{2_m^*} \int_{B\setminus B_{\varepsilon^\ga}} |u_\varepsilon^*|^{\2ms} dx\\
&\leq 2^n C^{2_m^*} \int_{\R^n \setminus B_{\varepsilon^\ga}} (1+ |x|^2)^{-n} dx\\
&= O (\varepsilon^{n(1-\gamma)})_{\varepsilon \searrow 0}.
\end{align*}
Thus, we have proved that
\begin{equation}\label{benngoai}
\int_{B\setminus B_{\varepsilon^\ga}} |v_\varepsilon|^{\2ms + |x|^\alpha} dx = O (\varepsilon^{n(1-\gamma)})_{\varepsilon \searrow 0}.
\end{equation}
On $B_{\varepsilon^\gamma} \setminus B_{a_\varepsilon}$, we also have $v_\varepsilon \leq 1$ and therefore
\[
v_\varepsilon(x) = A_{n,m} \frac{\varepsilon^{\frac{n-2m}2}}{(\varepsilon^2 + |x|^2)^{\frac{n-2m}2}} \geq A_{n,m} \varepsilon^{(1-2\gamma)\frac{n-2m}2} (1 + \varepsilon^{2(1-\gamma)})^{-\frac{n-2m}2},
\]
which implies
\begin{align*}
0\geq \ln  \big( v_\varepsilon(x)^{|x|^{\alpha}} \big) &\geq |x|^{\al}\ln  \lt(A_{n,m} \varepsilon^{(1-2\gamma)\frac{n-2m}2} (1 + \varepsilon^{2(1-\gamma)})^{-\frac{n-2m}2}\rt)\\
& \geq \varepsilon^{\gamma \alpha} \ln\lt(A_{n,m} \varepsilon^{(1-2\gamma)\frac{n-2m}2} (1 + \varepsilon^{2(1-\gamma)})^{-\frac{n-2m}2}\rt)\\
&= (1-2\gamma)\frac{n-2m}2 \varepsilon^{\gamma \alpha} \ln \varepsilon + O(\varepsilon^{\gamma \al})_{\varepsilon \searrow 0}\\
& = o(1)_{\varepsilon \searrow 0}.
\end{align*}
Hence, together with Taylor's expansion, we are in position to estimate $v_\varepsilon (x)^{|x|^\alpha}$ as follows
\begin{align*}
v_\varepsilon(x)^{|x|^\alpha} &= 1 + \Big( \ln A_{n,m} -\frac{n-2m}2 \ln \big(\varepsilon + \frac{|x|^2}{\varepsilon} \big) \Big) |x|^\al\\
&\quad + O\Big(\Big( \ln A_{n,m} -\frac{n-2m}2 \ln \big(\varepsilon + \frac{|x|^2}{\varepsilon} \big) \Big) ^2 |x|^{2\alpha}\Big)_{\varepsilon \searrow 0}
\end{align*}
on $B_{\varepsilon^{\gamma}} \setminus B_{a_\varepsilon}$. For $\varepsilon >0$ small enough, there holds
\begin{align}\label{eq:adgoutball}
\int_{B_{\varepsilon^{\gamma}} \setminus B_{a_\varepsilon}} & |v_\varepsilon|^{\2ms + |x|^\alpha} dx \notag \\
&= \int_{B_{\varepsilon^{\gamma}} \setminus B_{a_\varepsilon}} |v_\varepsilon|^{\2ms} dx + \Big( \ln A_{n,m}-\frac{n-2m}2 \ln \varepsilon \Big) \int_{B_{\varepsilon^{\gamma}} \setminus B_{a_\varepsilon}} |v_\varepsilon|^{\2ms} |x|^\alpha dx \notag\\
& -\frac{n-2m}2 \int_{B_{\varepsilon^{\gamma}} \setminus B_{a_\varepsilon}} |v_\varepsilon|^{\2ms} \ln \Big( 1 + \frac{|x|^2}{\varepsilon^2} \Big) |x|^\alpha dx \notag\\
&+ O\Big(\int_{B_{\varepsilon^{\gamma}} \setminus B_{a_\varepsilon}} |v_\varepsilon|^{\2ms} \Big( \ln A_{n,m} -\frac{n-2m}2 \ln \big(\varepsilon + \frac{|x|^2}{\varepsilon} \big) \Big) ^2 |x|^{2\alpha}  dx\Big)_{\varepsilon \searrow 0}.
\end{align}
On the other hand, there holds
\begin{align*}
v_\varepsilon(x)^{|x|^\alpha} &= A_{n,m}^{|x|^\alpha} \Big( \frac{\varepsilon}{\varepsilon^2 + |x|^2} \Big) ^{\frac{n-2m}{2}|x|^\alpha} \\
&= \exp \Big[\Big( \ln A_{n,m} -\frac{n-2m}2 \ln \big(\varepsilon + \frac{|x|^2}{\varepsilon} \big) \Big) |x|^\alpha\Big].
\end{align*}
Since $\eta =1$ on $B_{a_\varepsilon}$ for $\varepsilon >0$ small enough, we deduce that $v_\varepsilon = C u^*_\varepsilon \geq 1$ on $B_{a_\varepsilon}$. From this one can estimate
\begin{align*} 
0 &\leq \Big( \ln A_{n,m} -\frac{n-2m}2 \ln (\varepsilon + \frac{|x|^2}{\varepsilon}) \Big) |x|^\alpha \\
&\leq \Big( \ln A_{n,m} -\frac{n-2m}2 \ln \varepsilon \Big) a_\varepsilon^\alpha \\
&= O(\varepsilon^{\alpha /2}(-\ln \varepsilon))\\
& = o(1)_{\varepsilon \searrow 0}.
\end{align*}
Hence, together with Taylor's expansion, we are in position to estimate $v_\varepsilon (x)^{|x|^\alpha}$ as follows
\begin{align*}
v_\varepsilon(x)^{|x|^\alpha} &= 1 + \Big( \ln A_{n,m} -\frac{n-2m}2 \ln (\varepsilon + \frac{|x|^2}{\varepsilon}) \Big) |x|^\al\\
&\quad +O\Big( \Big( \ln A_{n,m} -\frac{n-2m}2 \ln (\varepsilon + \frac{|x|^2}{\varepsilon}) \Big) ^2 |x|^{2\alpha}\Big)_{\varepsilon \searrow 0},
\end{align*}
on $B_{a_\varepsilon}$. For $\varepsilon >0$ small enough, it holds
\begin{align}\label{eq:adg}
\int_{B_ {a_\varepsilon}}    |v_\varepsilon|^{\2ms + |x|^\alpha} dx 
=& \int_{B_ {a_\varepsilon}} |v_\varepsilon|^{\2ms} dx + \Big( \ln A_{n,m}-\frac{n-2m}2 \ln \varepsilon \Big) \int_{B_{a_\varepsilon}} |v_\varepsilon|^{\2ms} |x|^\alpha dx \notag\\
& -\frac{n-2m}2 \int_{B_ {a_\varepsilon}} |v_\varepsilon|^{\2ms} \ln \Big( 1 + \frac{|x|^2}{\varepsilon^2} \Big) |x|^\alpha dx \notag\\
&+ O\Big( \int_{B_{a_\varepsilon}} |v_\varepsilon|^{\2ms} \Big( \ln A_{n,m} -\frac{n-2m}2 \ln (\varepsilon + \frac{|x|^2}{\varepsilon}) \Big) ^2 |x|^{2\alpha}  dx\Big)_{\varepsilon \searrow 0}.
\end{align}
Combining \eqref{benngoai}, \eqref{eq:adgoutball} and \eqref{eq:adg}, we arrive
\begin{align}\label{eq:adgtong}
\int_{B} |v_\varepsilon|^{\2ms + |x|^\alpha} dx =& \int_{B_{\varepsilon^\ga}} |v_\varepsilon|^{\2ms} dx + \Big( \ln A_{n,m}-\frac{n-2m}2 \ln \varepsilon \Big) \int_{B_{\varepsilon^\ga}} |v_\varepsilon|^{\2ms} |x|^\alpha dx  \notag\\
& -\frac{n-2m}2 \int_{B_{\varepsilon^\ga}} |v_\varepsilon|^{\2ms} \ln \Big( 1 + \frac{|x|^2}{\varepsilon^2} \Big) |x|^\alpha dx+ O (\varepsilon^{n(1-\gamma)})_{\varepsilon \searrow 0} \notag\\
&+ O\Big(\int_{B_{\varepsilon^{\gamma}} } |v_\varepsilon|^{\2ms} \Big( \ln A_{n,m} -\frac{n-2m}2 \ln (\varepsilon + \frac{|x|^2}{\varepsilon}) \Big) ^2 |x|^{2\alpha}  dx\Big)_{\varepsilon \searrow 0}.
\end{align}
Now we estimate all integrals on the right hand side of \eqref{eq:adgtong}. It is easy to check that
\begin{align}\label{eq:dg2}
\int_{B_{\varepsilon^{\gamma}}} |v_\varepsilon|^{\2ms} dx = \int_{B_{\varepsilon^\ga}} (C u_\varepsilon^*)^{2_m^*} dx = C^{2_m^*} \Snm^{-\frac nm}+ O(\varepsilon^{n(1-\gamma)})_{\varepsilon \searrow 0}.
\end{align}
For $\beta \geq 0$ and $\de \geq 0$, we claim that
\begin{align}\label{eq:claimtiemcan}
\int_{B_{\varepsilon^\gamma}}& |v_\varepsilon|^{\2ms} |x|^{\beta} \lt(\ln \lt(1 + \frac{|x|^2}{\varepsilon^2}\rt)\rt)^{\de} dx\notag\\
&=
\left\{
\begin{aligned}
& \varepsilon^\beta A_{n,m}^{2_m^*} \int_{\R^n} \frac{|x|^\beta}{(1+ |x|^2)^n} \lt(\ln(1 + |x|^2)\rt)^\de dx + o(\varepsilon^\beta)&\mbox{ if }\beta < n,\\
&  \frac{(2(1-\gamma))^{\de+1}}{2(1 + \de)}A_{n,m}^{2_m^*}\om_{n-1} \varepsilon^n (-\ln \vep)^{1+ \de} + o(\varepsilon^n(-\ln \vep)^{1+\de}) &\mbox{ if } \beta =n,\\
& \frac{(2(1-\gamma))^\de}{\beta -n} A_{n,m}^{2_m^*}\om_{n-1}\varepsilon^{\gamma \beta + n(1-\gamma)}(-\ln \vep)^\de + o(\varepsilon^{\gamma \beta + n(1-\gamma)})&\mbox{ if } \beta > n.
\end{aligned}
\right.
\end{align}
Indeed, recall that
\[
v_\varepsilon (x) = A_{n,m} \eta(x) \varepsilon^{-\frac{n-2m}2} \Big(1 + \frac{|x|^2}{\varepsilon^2}\Big)^{-\frac{n-2m}2}.
\]
Hence, for $\varepsilon >0$ small enough, we have
\[
v_\varepsilon (x) = A_{n,m} \varepsilon^{-\frac{n-2m}2} \Big(1 + \frac{|x|^2}{\varepsilon^2} \Big)^{-\frac{n-2m}2}
\]
on $B_{\varepsilon^\ga}$. Making use of a suitable change of variables, we have
\[
\int_{B_{\varepsilon^\gamma}} |v_\varepsilon|^{\2ms} |x|^{\beta} \Big(\Big(1 + \frac{|x|^2}{\varepsilon^2}\Big)\Big)^{\de} dx= A_{n,m}^{\2ms} \varepsilon^\beta \int_{B_{\varepsilon^{\ga -1}}} \frac{|x|^\beta}{(1+ |x|^2)^n} \big(\ln (1 + |x|^2)\big)^\de dx.
\]
If $\beta < n$, then the function $ |x|^\beta (1+ |x|^2)^{-n} \big(\ln (1 + |x|^2)\big)^\de$ is integrable over $\R^n$, which then implies the first case in \eqref{eq:claimtiemcan}. If $\beta =n$, we have 
\[
\int_{B_{\varepsilon^{\ga -1}}} \frac{|x|^n}{(1+ |x|^2)^n} \lt(\ln (1 + |x|^2)\rt)^\de dx =\om_{n-1}\int_0^{\varepsilon^{\ga -1}} \frac{r^{2n-1}}{(1+ r^2)^n} (\ln (1+ r^2))^\de dr.
\]
By the l'H\^opital rule, we easily check that 
\[
\lim_{\varepsilon \searrow 0} \Big[ \frac{1}{(-\ln \vep)^{\de +1}} \int_0^{\varepsilon^{\ga -1}} \frac{r^{2n-1}}{(1+ r^2)^n} (\ln (1+ r^2))^\de dr \Big]= \frac{(2(1-\ga))^{\de +1}}{2(1+\de)},
\]
which proves the second case in \eqref{eq:claimtiemcan}. If $\beta > n$, we have
\[
\int_{B_{\varepsilon^{\ga -1}}} \frac{|x|^\beta}{(1+ |x|^2)^n} \lt(\ln (1 + |x|^2)\rt)^\de dx =\om_{n-1}\int_0^{\varepsilon^{\ga -1}} \frac{r^{\beta +n -1}}{(1+ r^2)^n} (\ln (1+ r^2))^\de dr.
\]
Again, by the l'H\^opital rule, we can also check that 
\[
\lim_{\varepsilon \searrow 0} \Big[ \frac{1}{\varepsilon^{(\be -n)(\ga -1)} (-\ln \vep)^\de} \int_0^{\varepsilon^{\ga -1}} \frac{r^{\beta+ n-1}}{(1+ r^2)^n} (\ln (1+ r^2))^\de dr \Big] = \frac{(2(1-\ga))^{\de}}{\beta -n},
\]
which proves the third case in \eqref{eq:claimtiemcan}. We now use \eqref{eq:claimtiemcan} to get
\begin{align}\label{eq:dg3}
&\Big(\ln A_{n,m} -\frac{n-2m}2 \ln \vep\Big)\int_{B_{a_\varepsilon}} |v_\varepsilon|^{\2ms} |x|^\alpha dx\notag\\
&= 
\left\{
\begin{aligned}
&\frac{n-2m}2 A_{n,m}^{\2ms}\int_{\R^n} \frac{|x|^\al}{(1+ |x|^2)^{n}} dx\, \varepsilon^{\alpha}|\ln \vep| + o(\varepsilon^{\al}|\ln \vep|)_{\varepsilon \searrow 0}&\mbox{ if } \al < n,\\
&\frac{n-2m}2(1-\gamma) A_{n,m}^{\2ms} \varepsilon^n (\ln \vep)^2 + o(\varepsilon^n (\ln \vep)^2)_{\varepsilon \searrow 0}&\mbox{ if } \al =n,\\
&\frac{n-2m}{2(\al -n)}A_{n,m}^{\2ms}\om_{n-1} \varepsilon^{\gamma \al + n(1-\gamma)} |\ln \vep| + o(\varepsilon^{\gamma \al + n(1-\gamma)}|\ln \vep|)_{\varepsilon \searrow 0}&\mbox{ if }\al > n.
\end{aligned}
\right.
\end{align}
Similarly, we have
\begin{align}\label{eq:dg4}
&\int_{B_ {a_\varepsilon}} |v_\varepsilon|^{\2ms} \ln \Big( 1 + \frac{|x|^2}{\varepsilon^2} \Big) |x|^\alpha dx\notag\\
& =
\left\{
\begin{aligned}
&\varepsilon^{\alpha}A_{n,m}^{\2ms} \int_{\R^n} \frac{|x|^\al}{(1+ |x|^2)^{n}} \ln (1 + |x|^2)  dx + o(\varepsilon^{\al})_{\varepsilon \searrow 0}&\mbox{ if } \al < n,\\
&(1-\gamma)^2 A_{n,m}^{\2ms} \om_{n-1}\varepsilon^n (\ln \vep)^2 + o(\varepsilon^n (\ln \vep)^2)_{\varepsilon \searrow 0} &\mbox{ if } \al = n,\\
&\frac{2(1-\ga)}{\al -n}A_{n,m}^{\2ms} \om_{n-1}\varepsilon^{\ga \al + n(1-\ga)} |\ln \vep| + o(\varepsilon^{\ga \al + n(1-\ga)}|\ln \vep|)_{\varepsilon \searrow 0} &\mbox{ if } \al > n.
\end{aligned}
\right.
\end{align}
By writing 
\[
\ln A_{n,m} -\frac{n-2m}2 \ln (\varepsilon + \frac{|x|^2}{\varepsilon}) =  \ln A_{n,m} -\frac{n-2m}2 \ln \varepsilon  -\frac{n-2m}2 \ln \big(1 + \frac{|x|^2}{\varepsilon^2} \big)
\]
and expanding $( \ln A_{n,m} -\frac{n-2m}2 \ln (\varepsilon + \frac{|x|^2}{\varepsilon}))^2$ and using again \eqref{eq:claimtiemcan} we have
\begin{align}\label{eq:dg5}
\int_{B_{\varepsilon^{\gamma}} } |v_\varepsilon|^{\2ms} \Big( \ln A_{n,m} &-\frac{n-2m}2 \ln (\varepsilon + \frac{|x|^2}{\varepsilon}) \Big) ^2 |x|^{2\alpha}  dx\notag\\
&=
\left\{
\begin{aligned}
& O(\varepsilon^{2\al} (\ln \vep)^2)_{\varepsilon \searrow 0} &\mbox{ if }\al <  n/2,\\
& O(\varepsilon^n |\ln \vep|^3)_{\varepsilon \searrow 0}&\mbox{ if }\al =  n/2,\\
& O(\varepsilon^{2\ga \al + n(1-\ga)}(\ln \vep)^2)_{\varepsilon \searrow 0}&\mbox{ if } \al > n/2.
\end{aligned}
\right.
\end{align}
Collecting all estimates \eqref{eq:adgtong} \eqref{eq:dg2}, \eqref{eq:dg3}, \eqref{eq:dg4} and \eqref{eq:dg5} gives
\begin{align*}
\int_{B} |v_\varepsilon|^{\2ms + |x|^\alpha} dx &\geq  C^{\2ms}\Snm^{-n/m} + \lt(\frac{n-2m}2 A_{n,m}^{\2ms}\int_{\R^n} \frac{|x|^\al}{(1+ |x|^2)^{n}} dx\rt) \varepsilon^{\alpha}(-\ln \vep)\\
& \quad+ o(-\varepsilon^{\al} \ln \vep)_{\varepsilon \searrow 0} + O(\varepsilon^{n(1-\gamma)})_{\varepsilon \searrow 0}
\end{align*}
if $\al < n$. Choosing $\gamma >0$ small enough so that $n(1-\ga) > \alpha$, we obtain \eqref{eq:tiemcan3} with
\[
\mathscr C_1 = \frac{n-2m}2 2^n\int_{\R^n} \frac{|x|^\al}{(1+ |x|^2)^{n}} dx
\] 
for $\al < n$. It also follows from \eqref{eq:adgtong} \eqref{eq:dg2}, \eqref{eq:dg3}, \eqref{eq:dg4} and \eqref{eq:dg5} that
\[
\int_{B} |v_\varepsilon|^{\2ms + |x|^\alpha} dx \geq C^{\2ms}\Snm^{-n/m} + \ga \frac{n-2m}2 A_{n,m}^{\2ms}\varepsilon^{n}(\ln \vep)^2 + O(\varepsilon^{n(1-\gamma)})_{\varepsilon \searrow 0}
\]
if $\al =n$, and finally
\[
\int_{B} |v_\varepsilon|^{\2ms + |x|^\alpha} dx \geq C^{\2ms}\Snm^{-n/m} + \frac{n-2m}{2(\al -n)}A_{n,m}^{\2ms} \varepsilon^{\gamma \al + n(1-\gamma)} (-\ln \varepsilon) + O(\varepsilon^{n(1-\gamma)})_{\varepsilon \searrow 0}
\]
if $\al > n$. This proves \eqref{eq:tiemcan3} for $\al \geq n$.
\end{proof}

Our next result in this subsection concerns the expansion of $\int_B \frac{1}{\2ms + |x|^{\alpha}} |u_\varepsilon(x)|^{\2ms + |x|^\alpha} dx$ in terms of $\varepsilon$.

\begin{lemma}\label{dgiatp}
Let $\alpha> 0$ and denote
\[
u_\varepsilon (x)= \eta (x) u_\varepsilon^* (x).
\] 
Then, there holds
\begin{equation}\label{eq:II}
\int_B \frac{|u_\varepsilon(x)|^{\2ms + |x|^\alpha}}{\2ms + |x|^{\alpha}}  dx 
= 
\left\{
\begin{aligned}
& \frac1{\2ms} \Snm^{-n/m} +\frac {\mathscr C_1}{\2ms} |\ln \varepsilon| \varepsilon^{\alpha} + O(\varepsilon^\alpha)_{\varepsilon \searrow 0} &\mbox{ if } \alpha < n, \\
& \frac1{\2ms} \Snm^{-n/m} + O(\varepsilon^{n(1-\ga)})_{\varepsilon \searrow 0}&\mbox{ if } \al \geq n,
\end{aligned}
\right.
\end{equation}
for any $0< \ga < 1/2$.
\end{lemma}

\begin{proof}
To proceed, we note that $u_\varepsilon^*(x) \geq 1$ if and only if
\[
|x| \leq ( 2\varepsilon -\varepsilon^2 ) ^{1/2} =: b_\varepsilon.
\] 
Notice also that $b_\varepsilon \sim \sqrt \varepsilon $ as $\varepsilon \to 0$. For any $0< \gamma < 1/2$, we have $b_\varepsilon < \varepsilon^\ga$ for $\varepsilon >0$ small enough.
\begin{equation}\label{eq:IItach}
\begin{aligned}
\int_B \frac{1}{\2ms + |x|^{\alpha}}  |u_\varepsilon(x)|^{\2ms + |x|^\alpha} dx &=\frac1{\2ms} \int_{B} |u_\varepsilon(x)|^{\2ms + |x|^\alpha} dx\\
&\quad  -\frac1{\2ms} \int_{B} \frac{|x|^\al}{\2ms + |x|^{\alpha}} |u_\varepsilon(x)|^{\2ms + |x|^\alpha} dx.
\end{aligned}
\end{equation}
Since $u_\varepsilon \leq 1$ on $B \setminus B_{\varepsilon^\ga}$, we have
\begin{align*}
0 &\leq \int_{B\setminus B_{\varepsilon^\ga}} \frac{|x|^\al}{\2ms + |x|^{\alpha}} |u_\varepsilon(x)|^{\2ms + |x|^\alpha} dx \\
&\leq \frac1{\2ms} \int_{B\setminus B_{\varepsilon^\ga}} |x|^\al |u_\varepsilon^*(x)|^{\2ms + |x|^\alpha} dx\\
&\leq \frac1{\2ms} 2^n \varepsilon^\al \int_{B_{\varepsilon^{-1}}\setminus B_{\varepsilon^{\ga-1}}} |x|^\al (1+ |x|^2)^{-n} dx.
\end{align*}
By the direct computations, we have
\begin{equation}\label{eq:benngoai}
\int_{B\setminus B_{\varepsilon^\ga}} \frac{|x|^\al}{\2ms + |x|^{\alpha}} |u_\varepsilon(x)|^{\2ms + |x|^\alpha} dx =
\left\{
 \begin{aligned}
& O(\varepsilon^{n(1-\gamma) + \al \ga})_{\varepsilon \searrow 0} &\mbox{ if } \al < n,\\
& O(-\varepsilon^{n} \ln \vep)_{\varepsilon \searrow 0}&\mbox{ if } \al =n,\\
& O(\varepsilon^n)_{\varepsilon \searrow 0} &\mbox{ if }\al > n.
\end{aligned}
\right.
\end{equation}
Repeating the proof of Lemma \ref{Tiemcan}, we have
\begin{align*}
0 & \leq \int_{B_{\varepsilon^\ga}}\frac{|x|^\al |u_\varepsilon|^{\2ms + |x|^\alpha}}{2_m^* + |x|^\al} dx\notag \\
&\leq \frac1{\2ms}\int_{B_{\varepsilon^\ga}}|x|^\al |u_\varepsilon|^{\2ms + |x|^\alpha} dx \notag\\
&=\frac1{\2ms} \int_{B_{\varepsilon^\ga}} |x|^\al |u_\varepsilon|^{\2ms} dx +\frac{n-2m}{2\2ms} \ln \frac2\varepsilon \int_{B_{\varepsilon^\ga}} |u_\varepsilon|^{\2ms} |x|^{2\alpha} dx \notag\\
& \quad -\frac{n-2m}{2\2ms} \int_{B_{\varepsilon^\ga}} |u_\varepsilon|^{\2ms} \ln \Big( 1 + \frac{|x|^2}{\varepsilon^2} \Big) |x|^{2\alpha} dx \notag\\
&\quad + O\Big( \int_{B_{\varepsilon^{\gamma}} } |u_\varepsilon|^{\2ms} \Big( \ln \frac2\varepsilon - \ln \big(1 + \frac{|x|^2}{\varepsilon^2} \big) \Big) ^2 |x|^{3\alpha}  dx\Big)_{\varepsilon \searrow 0}.
\end{align*}
The claim \eqref{eq:claimtiemcan} implies
\begin{align}\label{eq:TCII}
\int_{B_{\varepsilon^\ga}}\frac{|x|^\al |u_\varepsilon|^{\2ms + |x|^\alpha}}{2_m^* + |x|^\al} dx = 
\left\{
\begin{aligned}
& O(\varepsilon^\al) &\mbox{ if } \al < n,\\
& O(-\varepsilon^n \ln \varepsilon) &\mbox{ if } \al =n,\\
& O(\varepsilon^{\ga \al + n(1-\ga)}) &\mbox{ if } \al >n.
\end{aligned}
\right.
\end{align}
Combining \eqref{eq:benngoai} and \eqref{eq:TCII}, we get
\begin{equation}\label{phanhieu}
\int_{B} \frac{|x|^\al}{\2ms + |x|^{\alpha}} |u_\varepsilon(x)|^{\2ms + |x|^\alpha} dx = 
\left\{
\begin{aligned}
& O(\varepsilon^{\al})_{\varepsilon \searrow 0} &\mbox{ if } \al < n,\\
& O(-\varepsilon^{n} \ln \vep)_{\varepsilon \searrow 0}&\mbox{ if }\al =n,\\
& O(\varepsilon^n)_{\varepsilon \searrow 0} &\mbox{ if }\al > n.
\end{aligned}
\right.
\end{equation}
Inserting \eqref{phanhieu} and \eqref{eq:tiemcan3} with $C =1$ into \eqref{eq:IItach}, we obtain \eqref{eq:II}.
\end{proof}

Finally, let us recall a Brezis--Lieb lemma in the variable exponent Lebesgue spaces; see \cite[Lemma $3.4$]{BFS}.

\begin{lemma}\label{BL}
Let $f_j \to f$ a.e. and $f_j \rightharpoonup f$ weakly in $L_{p(x)}$, then
\[
\int |f_j|^{p(x)} dx = \int |f_j -f|^{p(x)}dx + \int |f(x)|^{p(x)} dx + o(1)_{j \nearrow +\infty}.
\]
\end{lemma}


\section{The supercritical Sobolev type inequality}

\subsection{The existence of the sharp constant $\Unm$: Proof of Theorem \ref{Supercritical}}

Instead of proving Theorem \ref{Supercritical} for the exponent $\2ms+|x|^\alpha$, we shall prove a more general result for the exponent
\[
\2ms + f(r),
\] 
where $f$ is a function satisfying the assumptions ($f_1$) and ($f_2$) from Lemma \ref{fract}. In this sense, we are about to show that
\[
\sup\Big\{\int_B |u(x)|^{\2ms + f(|x|)} dx   :  u\in \Hmr,  \|\nabla^m u\|_{L^2(B)} \leq 1\Big\} < + \infty.
\]
Furthermore, because the case $m=1$ was already considered in \cite{doO}, we
do not treat the case $m=1$ here. Instead, we only consider the case $m \geq 2$. Let $u \in \Cr$ such that $\|\nabla^m u\|_{L^2(\Om)} \leq 1$. The Hardy--Rellich inequality \eqref{eq:HR} tells us that 
\[
\int_0^1 (u'(r))^2 r^{n-2m+ 1} dr \leq \frac{1}{\CHR \omega_{n-1}}.
\]
Define 
$
w(s) = u(s^{1/m}).
$ 
Clearly, we have 
\[
w'(s) = \frac1m s^{\frac1m -1 } u'(s^{1/m}).
\]
Therefore
\[
\int_0^1 (u'(r))^2 r^{n-2m+ 1} dr = m^2 \int_0^1 (w'(s))^2 s^{\frac nm -1} ds.
\]
Hence, on one hand we obtain
\[
\int_0^1 (w'(s))^2 s^{\frac nm -1} ds \leq \frac{1}{m \CHR \omega_{n-1}} =: a.
\]
On the other hand, by making the change of variable $r =s^{1/m}$, we get
\[
\int_0^1 |u(r)|^{\2ms + f(r)} r^{n-1} dr =\frac1m \int_0^1 |w(s)|^{\2ms + f(s^{1/m})} s^{\frac nm-1} ds.
\]
Note that the function $g: s \mapsto f(s^{1/m})$ still satisfies the assumptions ($f_1$) and ($f_2$) of Lemma \ref{fract}. We next apply Lemma \ref{fract} for $\beta = n/m >2$ and the function $g$ to get
\[
\int_0^1 |w(s)|^{\2ms + g(s)} s^{\frac nm -1} ds \leq S_{\frac nm}(g,a),
\]
which implies
\[
\int_B |u(x)|^{\2ms + f(|x|)} dx = \omega_{n-1}\int_0^1 |u(r)|^{\2ms + f(r)} r^{n -1} dr \leq \frac{\omega_{n-1}}{m} S_{\frac nm}(g,a).
\]
This finishes the proof of this Theorem.


\subsection{The embedding $\Hmr \hookrightarrow L_{\2ms + |x|^\alpha}(B)$: Proof of Corollary \ref{varexspace}}

As in the preceding subsection, we can also enhance Corollary \ref{varexspace} by replacing the exponent $\2ms+|x|^\alpha$ by the following general exponent
\[
\2ms + f(r),
\] 
where, again, $f$ is a function satisfying the assumptions ($f_1$) and ($f_2$) from Lemma \ref{fract}. Unlike the proof of Theorem \ref{Supercritical}, our argument below also works for the case $m=1$. Let $u \in \Hmr\not\equiv 0$ be arbitrary and define $v = u/\|\nabla^m u\|_{L^2(B)}$. We clearly have $\|\nabla^m v\|_{L^2(B)} =1$. Then Theorem \ref{Supercritical} implies
\[
\int_B \Big|\frac{u(x)}{\|\nabla^m u\|_{L^2(B)}}\Big|^{\2ms + f(|x|)} dx \leq C 
\]
for some constant $C >0$ independent of $u$. This shows that $u \in L_{\2ms +f(|x|)}(B)$. Taking $\lambda_* \gg 1$ such that $C \lambda_*^{-\2ms} \leq 1$. Then we can estimate
\begin{align*}
\int_B \Big| \frac{u(x)}{\lambda_* \|\nabla^m u\|_{L^2(B)}}\Big|^{\2ms +f(|x|)} dx &= \lambda_*^{-\2ms} \int_B \Big|\frac{u(x)}{\|\nabla^m u\|_{L^2(B)}}\Big|^{\2ms + f(|x|)} \lambda_*^{-f(|x|)} dx\\
& \leq \lambda_*^{-\2ms} \int_B \Big|\frac{u(x)}{\|\nabla^m u\|_{L^2(B)}}\Big|^{\2ms + f(|x|)} dx \\
&\leq C \lambda_*^{-\2ms} \\
&\leq 1.
\end{align*}
By the definition of the norm $\|\cdot\|_{L_{\2ms + f(|x|)}(B)}$, we get
\[
\|u\|_{L_{\2ms + f(|x|)}(B)} \leq \lambda_* \|\nabla^m u\|_{L^2(B)}.
\]
This inequality proves the continuity of the embedding 
\[
\Hmr \hookrightarrow L_{\2ms + f(|x|)}(B).
\] 
In particular, given $\alpha > 0$, the embedding $\Hmr \hookrightarrow L_{\2ms + |x|^\alpha}(B)$ is continuous. 


\subsection{The inequality $\Unm \geq \Smn$: Proof of Theorem \ref{Strict}}
 
Assume that $1 \leq m < n/2$ and $\alpha > 0$. Let us define
\[
\bar u_\varepsilon(x) = \Snm^{n/(2m)} \eta(x) u_\varepsilon^*(x).
\]
We then have from \eqref{eq:tiemcan1} and \eqref{eq:tiemcan2} that
\begin{equation}\label{eq:tiemcan11}
\Big( \int_B |\nabla^m \bar u_\varepsilon|^2 dx \Big)^{1/2} = 1 + O(\varepsilon^{n-2m})_{\varepsilon \searrow 0},
\end{equation}
and
\begin{equation*}\label{eq:tiemcan21}
\int_B |\bar u_\varepsilon |^{\2ms} dx =  \Smn  + O(\varepsilon^n)_{\varepsilon \searrow 0}.
\end{equation*}
In view of \eqref{eq:tiemcan11}, there exists $C > 0$ such that
\[
0 < 1- C\varepsilon^{n-2m} \leq \|\nabla^m \bar u_\varepsilon\|_{L^2(B)} \leq 1 + C \varepsilon^{n-2m} 
\]
for $\varepsilon >0$ small enough. Hence, for some constant $C'>0$ and $\varepsilon >0$ small enough, there holds
\begin{align*}
\|\nabla^m \bar u_\varepsilon\|_{L^2(B)}^{\2ms+ |x|^\alpha} &\leq (1 + C \varepsilon^{n-2m})^{\2ms + |x|^\alpha}\\
& \leq (1 + C \varepsilon^{n-2m})^{\2ms+1} \leq 1 + C'\varepsilon^{n-2m} 
\end{align*}
everywhere on $B$. Similar, for some constant $C'>0$ and $\varepsilon >0$ small enough, we have
\begin{align*}
\|\nabla^m \bar u_\varepsilon\|_{L^2(B)}^{\2ms+ |x|^\alpha} & \geq (1 - C \varepsilon^{n-2m})^{\2ms + |x|^\alpha} \\
& \geq (1 - C \varepsilon^{n-2m})^{\2ms+1} \geq 1 -C''\varepsilon^{n-2m}
\end{align*}
on $B$. Consequently, we get
\[
\|\nabla^m \bar u_\varepsilon\|_{L^2(B)}^{\2ms+ |x|^\alpha} = 1 + O(\varepsilon^{n-2m})_{\varepsilon \searrow 0}
\]
on $B$ and hence 
\[
\|\nabla^m \bar u_\varepsilon\|_{L^2(B)}^{-\2ms- |x|^\alpha} = 1 + O(\varepsilon^{n-2m})_{\varepsilon \searrow 0} 
\]
everywhere on $B$. Now Lemma \ref{Tiemcan} with $C =\Snm^{n/(2m)}$ implies
\[
\int_B |\bar u_\varepsilon|^{\2ms + |x|^\alpha} dx \geq \Smn + o(1)_{\varepsilon \searrow 0}.
\]
Hence, there holds
\begin{align*}
\Unm&\geq \int_B \Big|\frac{\bar u_\varepsilon}{\|\nabla^m \bar u_\varepsilon\|_{L^2(B)}}\Big|^{\2ms + |x|^\alpha} dx\\
&= \int_B |\bar u_\varepsilon|^{\2ms + |x|^\alpha} \|\nabla^m \bar u_\varepsilon\|_{L^2(B)}^{-\2ms - |x|^\alpha}  dx\\
&= (1+ O(\varepsilon^{n-2m})_{\varepsilon \searrow 0}) \int_B |\bar u_\varepsilon|^{\2ms + |x|^\alpha} dx\\
&\geq (1+ O(\varepsilon^{n-2m})_{\varepsilon \searrow 0})(\Smn + o(1)_{\varepsilon \searrow 0}).
\end{align*}
Now we send $\varepsilon \searrow 0$ to conclude that
\[
\Unm \geq \Smn.
\]
This proves \eqref{U>=S}.

We now prove the strict inequality \eqref{U>S}. Suppose $0< \alpha \leq n -2m$. Applying Lemma \ref{Tiemcan} with $C =\Snm^{n/(2m)}$ gives
\begin{align*}
\Unm &\geq (1+ O(\varepsilon^{n-2m})_{\varepsilon \searrow 0}) \big[\Smn + \Smn^{n/(2m)} \mathscr C_1 |\ln \varepsilon| \varepsilon^\alpha + o(\varepsilon^\alpha|\ln \vep|)_{\varepsilon \to 0 }\big]\\
&=\Smn + \Smn^{n/(2m)} \mathscr C_1 |\ln \varepsilon| \varepsilon^\alpha + O(\varepsilon^{n-2m})_{\varepsilon \searrow 0}+ o(\varepsilon^\alpha|\ln \vep|)_{\varepsilon \searrow 0}\\
&> \Smn
\end{align*}
for $\varepsilon >0$ small enough since $\alpha \leq n -2m$.

Finally, we study the limit  of $\Unm$ as $\alpha \nearrow +\infty$. In view of \eqref{U>=S}, we always have
\[
\liminf_{\al\to +\infty} \Unm \geq \Smn.
\]
To finish the proof of \eqref{U-infinity}, we only have to check
\[
\limsup_{\al\to +\infty} \Unm \leq \Smn.
\]
By way of contradiction, there exists an increasing sequence $(\alpha_i)_{i \geq 1}$ with $\alpha_ i \nearrow +\infty$ such that
\begin{equation}\label{eq:tocontradict}
\lim_{i \to +\infty} \Unma > \Smn.
\end{equation}
(We may assume at the  beginning that $\alpha_1 > 1$, just for convenience.) For each $i$, from the definition of $\Unma$, we can choose $u_i \in \Hmr$ with $\|\na^m u_i\|_{L^2(B)} =1$ and with
\[
\int_{B} |u_i|^{\2ms + |x|^{\al_i}} dx \geq \Unma -\frac1i.
\]
Since $(u_i)_i$ is bounded in $\Hmr$, up to a subsequence which we still denote by $(u_i)_i$, we can assume that
\begin{itemize}
  \item $u_i \rightharpoonup u_0$ weakly in $\Hmr$ and
  \item $u_i \to u_0$ a.e. in $B$. 
\end{itemize} 
In view of \eqref{eq:HR} and because $\|\nabla^m u_i \|_{L^2(B)}=1$, we have
\[
\int_B \frac{|\na u_i|^2}{|x|^{2(m-1)}} dx \leq \frac{1}{\CHR},
\]
which implies
\[
\int_0^1 (u'_i(r))^2 r^{n-2m+1} dr \leq \frac1{\omega_{n-1} \CHR} =:C
\]
for any $j \geq 1$. Similar to the estimate \eqref{eq:est1} and by a simple density argument, we easily get for any $i \geq 1$ that
\begin{equation}\label{eq:est11}
|u_i(x)| \leq \Big( \frac C{n-2m} \Big) ^{1/2} |x|^{-\frac{n-2m}2}
\end{equation}
for a.e. in $B$. Consequently, by Lebesgue's dominated convergence theorem, we have
\[
\lim_{i \nearrow +\infty} \int_{B\setminus B_s} |u_i|^{\2ms + |x|^{\al_i}} dx = \lim_{i \nearrow +\infty} \int_{B\setminus B_s} |u_i|^{\2ms} dx   = \int_{B\setminus B_s} |u_0|^{\2ms} dx
\]
for any $ 0< s < 1$ but fixed. Thus, we can write
\begin{equation}\label{eq:tiemcanivocung}
\int_{B\setminus B_s} |u_i|^{\2ms + |x|^{\al_i}} dx =  \int_{B\setminus B_s} |u_i|^{\2ms} dx  + o_s(1)_{i \nearrow+\infty},
\end{equation}
here by $o_s(1)_{i \nearrow +\infty}$ we mean $\lim_{i \nearrow +\infty} o_s(1) =0$ for each $s$ fixed. Taking $s_0$ in such a way that
\[
\Big( \frac C{n-2m} \Big) ^{1/2} s_0^{-\frac {n-2m}2} =1.
\] 
Hence, on $B_s$ with $0< s < \min \{1, s_0\} $, thanks to $\alpha_i > 1$, we have
\[
|u_i(x)|^{|x|^{\al_i}} \leq \Big(\Big( \frac C{n-2m} \Big) ^{1/2}s^{-\frac{n-2m}2}\Big)^{s^{\al_i}} \leq \Big(\Big( \frac C{n-2m} \Big) ^{1/2} s^{-\frac{n-2m}2}\Big)^s.
\]
Since
\[
\lim_{s\to 0} \Big(\Big( \frac C{n-2m} \Big) ^{1/2} s^{-\frac{n-2m}2}\Big)^s = 1,
\] 
for any $\varepsilon >0$, we can choose $s < \min \{1, s_0\} $ such that 
\[
\Big(\Big( \frac C{n-2m} \Big) ^{1/2} s^{-\frac{n-2m}2}\Big)^s \leq 1+\vep.
\]
Therefore, we have
\[
\int_{B_s} |u_i|^{\2ms + |x|^{\al_i}} dx \leq (1+ \varepsilon)\int_{B_s} |u_i|^{\2ms} dx.
\]
This estimate together with \eqref{eq:tiemcanivocung} implies
\begin{align*}
\Unma-\frac1i  &\leq \int_{B} |u_i|^{\2ms + |x|^{\al_i}} dx \\
& \leq (1+ \varepsilon) \int_B |u_i|^{\2ms} dx + o_s(1)_{i \nearrow+\infty} \\
&\leq (1+ \varepsilon)\Smn + o_s(1)_{i \nearrow+\infty}.
\end{align*}
Letting $i\to +\infty$ and then $\varepsilon \to 0$, we obtain
\[
\lim_{i \nearrow +\infty} \Unma \leq \Smn,
\]
which contradicts \eqref{eq:tocontradict}. This contradiction completes the proof.

\subsection{The sharp constant $\Unm$ is attained: Proof of Theorem \ref{Attain}}

Assume that $1 \leq m < n/2$, that $\alpha > 0$, and that $\Unm > \Snm$. Let $(u_j)_j$ be a maximizing sequence for $\Unm$ in $\Hmr$. By normalizing $u_j$, if necessary, we may assume that $\|\nabla^m u_j \|_{L^2(B)}=1$. By the boundedness of $(u_j)_j$ in $\Hmr$ and the Sobolev embedding, there exists some function $u\in \Hmr$ such that
\begin{itemize}
 \item $u_j \hookrightarrow u$ weakly in $H_0^m(B)$,
 \item $u_j \to u$ strongly in $H_0^k(B)$ for any $0 \leq k < m$, and 
 \item $u_j \to u$ a.e. in $B$,
\end{itemize} 
as $j \to +\infty$. We claim that $u\not\equiv 0$. Indeed, suppose that $u\equiv 0$. Let $\eta$ be the cut-off function on $B$ used in the proof of Theorem \ref{Strict} and for $\delta >0$ we define
\[
\eta_\de(x) = \eta(x/\de).
\] 
We note that
\[
\nabla^m( \eta_\delta u_j) = \eta_\delta \nabla^m u_j + F_j,
\]
where $F_j$ is linear combination of derivatives of $u_j$ with order strictly less than $m$. Hence $\|F_j\|_{L^2(B)} \to 0$ by a compact embedding. Put
\[
a_{\de,j} = \|\nabla^m (\eta_\delta u_j)\|_{L^2(B)}.
\] 
Since $a_{\de,j} \leq \|\eta_\delta \nabla^m u_j\|_{L^2(B)} + \|F_j\|_{L^2(B)}$ and $0\leq \eta_\delta \leq 1$, we have
\begin{equation*}\label{eq:limsup2}
\limsup_{j \nearrow +\infty}a_{\de,j} \leq 1,
\end{equation*}
for any $\delta >0$. Since
\[
\lim_{r\to 0} \Big( \frac C{n-2m} r^{2m-n} \Big) ^{r^\alpha/2} =1,
\]
for arbitrary $\varepsilon >0$ but fixed, we can choose some $\delta >0$ in such a way that 
\[
 \Big( \frac C{n-2m} r^{2m-n} \Big) ^{r^\alpha/2} \leq 1+ \varepsilon
\]
for any $0 < r < \de$. Fix such a $\delta >0$, we have by \eqref{eq:est11} that
\begin{align*}
(\eta_\delta |u_j(x)|)^{\2ms + |x|^\alpha} &\leq (\eta_\delta |u_j(x)|)^{\2ms} \Big( \frac C{n-2m} |x|^{2m-n} \Big) ^{|x|^\alpha/2} \\
&\leq (1+ \varepsilon)(\eta_\delta |u_j(x)|)^{\2ms}.
\end{align*}
Integrating over $B$, we get
\begin{align}\label{eq:int1}
\int_B (\eta_\delta |u_j(x)|)^{\2ms + |x|^\alpha} dx &\leq (1+ \varepsilon)\int_B (\eta_\delta |u_j(x)|)^{\2ms} dx\notag\\
& \leq (1+ \varepsilon)\int_B |u_j(x)|^{\2ms} dx\notag\\
& \leq (1+ \varepsilon) \Smn,
\end{align}
where the last inequality follows from the characterization of the sharp constant $\Snm$. On the other hand, because $\eta_\delta = 1$ in $B_{\delta/2}$, we estimate
\[
\int_B \Big( |u_j(x)|^{\2ms + |x|^\alpha} -(\eta_\delta |u_j(x)|)^{\2ms + |x|^\alpha} \Big) dx \leq \int_{B\setminus B_{\delta/2}} |u_j(x)|^{\2ms + |x|^\alpha} dx.
\]
The estimate \eqref{eq:est11} show that $|u_j(x)|^{\2ms + |x|^\alpha}$ is uniformly bounded on $B\setminus B_{\delta/2}$. Moreover $u_j \to 0$ a.e on $B$. We are now able to apply Lebesgue's dominated convergence theorem to get
\begin{equation}\label{eq:zero1}
\begin{aligned}
\limsup_{j \nearrow +\infty} \int_B \Big( |u_j(x)|^{\2ms + |x|^\alpha} & -(\eta_\delta |u_j(x)|)^{\2ms + |x|^\alpha} \Big) dx \\
&\leq \lim_{j \nearrow +\infty}\int_{B\setminus B_{\delta/2}} |u_j(x)|^{\2ms + |x|^\alpha} dx =0.
\end{aligned}
\end{equation}
Putting \eqref{eq:int1} and \eqref{eq:zero1} together, we get 
\[
\Unm = \lim_{j \nearrow +\infty} \int_B |u_j(x)|^{\2ms + |x|^\alpha} dx \leq (1+ \varepsilon) \Smn,
\]
for arbitrary $\varepsilon >0$ but fixed. This contradicts to our assumption $\Unm > \Smn$ if we choose $\varepsilon >0$ small enough. Hence, $u\not\equiv 0$ as claimed.

In the rest of the proof, we show that $u$ is an optimizer for $\Unm$. Recall that the embedding $\Hmr \hookrightarrow L_{\2ms + |x|^\alpha}(B)$ is continuous by means of Corollary \ref{varexspace}, which implies that $u_j \hookrightarrow u$ weakly in $L_{\2ms + |x|^\alpha}(B)$. By Lemma \ref{BL} we have
\begin{align}\label{eq:applyBL}
\Unm &= \int_B |u_j(x)|^{\2ms + |x|^\alpha} dx + o(1)_{j \nearrow +\infty} \notag\\
&= \int_B |u_j(x)-u(x)|^{\2ms +|x|^\alpha} dx + \int_B |u(x)|^{\2ms + |x|^\alpha} dx + o(1)_{j \nearrow +\infty}.
\end{align}
Since $u_j \rightharpoonup u$ weakly in $H_0^m(B)$, we have
\[
1 = \|\nabla^m u_j\|_{L^2(B)}^2 = \|\nabla^m u_j -\nabla^m u\|_{L^2(B)}^2 + \|\nabla^m u\|_{L^2(B)}^2 + o(1)_{i \nearrow +\infty}.
\]
Put
\[
a = \|\nabla^m u\|_{L^2(B)} \in (0,1]
\]
and
\[
a_j = \|\nabla^m u_j -\nabla^m u\|_{L^2(B)}.
\] 
Depending on the size of $a$, we have the following two possible cases:

\noindent\textbf{Case 1}. Suppose that $a < 1$. Then, we have
\[
\lim_{j\to\infty} a_j = (1-a^2)^{1/2} \in (0,1).
\]
Hence for $j$ large enough, there holds $0 < a_j < 1$. From \eqref{eq:applyBL} and the definition of $L_{\2ms + |x|^\alpha}$-norm, we get
\begin{align*}
\Unm& = \int_B \Big( \frac{|u_j(x)-u(x)|}{a_j} \Big) ^{\2ms + |x|^\alpha} a_j^{\2ms+ |x|^\alpha} dx \\
& \quad + \int_B \Big( \frac{|u(x)|}{a} \Big) ^{\2ms + |x|^\alpha} a^{\2ms + |x|^\alpha} dx + o(1)_{j \nearrow +\infty}\\
&\leq a_j^{\2ms}\int_B \Big( \frac{|u_j(x)-u(x)|}{a_j} \Big) ^{\2ms + |x|^\alpha} dx \\
& \quad + a^{\2ms} \int_B \Big( \frac{|u(x)|}{a} \Big) ^{\2ms + |x|^\alpha} dx + o(1)_{j \nearrow +\infty}\\
&\leq \big( a_j^{\2ms} + a^{\2ms} \big) \Unm + o(1)_{j \nearrow +\infty}.
\end{align*}
Letting $j \nearrow +\infty$ and dividing both sides by $\Unm$ we get
\[
1 \leq (1-a^2)^{\2ms/2} + (a^2)^{\2ms/2},
\]
which is impossible since $\2ms >2$ and $0 < a < 1$. 

\noindent\textbf{Case 2}. As shown above, the only possible value for $a$ is that $a =1$, which yields
\[
u_j \to u
\]
strongly in $H_0^m(B)$. Again by Corollary \ref{varexspace}, now we have the following convergence
\[
b_j:= \|u_j -u\|_{L_{\2ms + |x|^\alpha}(B)} \to 0
\] 
as $j \nearrow +\infty$. By the definition of the $L_{\2ms + |x|^\alpha}$-norm, we have
\begin{align*}
\int_B |u_j(x)-u(x)|^{\2ms + |x|^\alpha} dx &=\int_B \Big|\frac{u_j(x)-u(x)}{b_j}\Big|^{\2ms + |x|^\alpha} b_j^{\2ms + x|^\alpha} dx \\
&\leq b_j^{\2ms} \int_B\Big|\frac{u_j(x)-u(x)}{b_j}\Big|^{\2ms + |x|^\alpha}dx \\
&\leq b_j^{\2ms}
\end{align*}
for $j$ large enough. Consequently, we have
\[
\lim_{j \nearrow +\infty} \int_B |u_j(x)-u(x)|^{\2ms + |x|^\alpha} dx = 0.
\]
We are now in position to pass \eqref{eq:applyBL} to the limit as $j\to +\infty$  to get 
\[
\int_B |u(x)|^{\2ms + |x|^\alpha} dx = \Unm.
\]
This shows that $u$ is indeed a maximizer for $\Unm$. The proof is complete.


\section{Higher order elliptic problems: Proof of Theorem \ref{Superequation}}

We now turn our attention to the existence result for solutions to \eqref{eq:supercriticaleq}, namely,
\[
\left\{
\begin{aligned}
(-\De)^m u = & u^{\2ms +|x|^\alpha -1} & \mbox{ in }& B,\\
u >& 0 & \mbox{ in } &B,\\
\pa^j_r u = &0 & \mbox{ on } &\partial B, \quad j =0,\ldots,m-1 .
\end{aligned}
\right.
\]
Since the above problem has a variational structure, we employ variational methods. To this purpose, we consider the functional
\[
I(u) = \frac12 \int_B |\nabla^m u|^2 dx - \int_B \frac{1}{\2ms + |x|^{\alpha}} u_+(x)^{\2ms + |x|^\alpha} dx
\]
on $\Hmr$. By Theorem \ref{Supercritical}, the functional $I$ is well-defined and of class $C^1$ on $\Hmr$. Consequently, if $u \in \Hmr$ is a critical point of $I $, namely,
\[
\langle I(u), \phi \rangle =   \int_B \langle \nabla^m u, \nabla^m \phi \rangle dx - \int_B (u_+ )^{\2ms + |x|^\alpha - 1} \phi dx
\]
for any $\phi \in \Cr$, it is not hard to see that $u$ solves \eqref{eq:supercriticaleq} weakly. In the rest of the proof, we shall show that $I$ admits a critical point in $\Hmr$, which is a saddle point. To this aim, we shall apply a variant of the well-known mountain pass theorem of Ambrosetti and Rabinowitz without the Palais--Smale condition. This result is due to Brezis and Nirenberg; see \cite[Theorem $2.2$]{BN}. For this reason and  following \cite{doO} closely, our strategy is to prove the following facts:
\begin{enumerate}
\item[(A)] The level $$\frac mn \Snm^{-n/m}$$ is a non-compactness level for the functional $I$; see Lemma \ref{lemNonCompactLevel}.

\item[(B)] The mountain-pass level $c$ of the functional $I$, given by \eqref{MPlevel} below, satisfies
\[
c < \frac mn \Snm^{-n/m};
\]
see Lemma \ref{lemUpperBound4Level}.

\item[(C)] There exists a weak, non-trivial solution $u$ at level
\[
0 < c < \frac m n\Snm^{-n/m};
\] 
see Lemma \ref{lemExistenceWeak}.
\end{enumerate}

For clarity, we split our proof into several lemmas as below.

\begin{lemma}\label{lemNonCompactLevel}
The level $(m/n) \Snm^{-n/m}$ is a non-compactness level for the functional $I$.
\end{lemma}

\begin{proof}
To proceed, we continue using the function $u_\varepsilon^*$. Let $\eta$ be the cut-off function on $B$ as before and, as always, set $u_\varepsilon = \eta u_\varepsilon^*$. Then, we have $u_\varepsilon \in \Hmr$. Lemma \ref{dgiatp} and \eqref{eq:tiemcan1} yield
\[
I(u_\varepsilon) = \frac12 \int_B |\nabla^m u_\varepsilon|^2 dx - \int_B \frac{1}{\2ms + |x|^{\alpha}} u_\varepsilon(x)^{\2ms + |x|^\alpha} dx \to \frac mn \Snm^{-n/m}.
\]
Furthermore, for any $r\in (0,1/2)$, we have $u_\varepsilon = u_\varepsilon^*$ and therefore
\[
\int_{B_r} |\nabla^m u_\varepsilon|^2 dx = \int_{B_r}|\nabla^m u_\varepsilon^*|^2 dx = \int_{B_{r/\varepsilon}}|\nabla^m u_1^*|^2 dx \to \Snm^{-n/m}
\]
as $\varepsilon \to 0$. Moreover, there holds
\[
\int_B u_\varepsilon(x)^2 dx \leq \int_B(u_\varepsilon^*(x))^2 dx = \varepsilon^{2m} \int_{B_{1/\varepsilon}} (1+ |x|^2)^{-n+ 2m} dx \to 0
\]
as $\varepsilon \to 0$. Hence, the sequence $u_\varepsilon$ is concentrating and converges weakly to $0$ in $\Hmr$. We have shown that the sequence $u_\varepsilon$ is concentrating at $0$, converges weakly to $0$ and does not contain a strongly convergent subsequence. 
\end{proof}

\begin{lemma}\label{lemMPstructure}
The functional $I$ has a mountain-pass structure in the sense of \cite[Theorem $2.2$]{BN}. 
\end{lemma}

\begin{proof}
To conclude the lemma, we have to show that, following the notations used in \cite{BN}, the two conditions (2.9) and (2.10) in \cite{BN} are satisfied for a suitable neighborhood $U$, a suitable constant $\rho>0$, and a suitable $v \in \Hmr$. For $U$, we simply take a ball in $\Hmr$ centered at zero with radius $\tau \ll 1$. Then on the boundary of $U$ in $\Hmr$, namely, those functions having $\|\nabla^m u\|_{L^2(B)}=\tau$, we can apply Theorem \ref{Supercritical} to get
\[
\big( \frac 1\tau \big)^{\2ms } \int_B \big| u(x) \big|^{\2ms + |x|^\alpha} dx \leq \int_B \big| \frac {u(x)}\tau\big|^{\2ms + |x|^\alpha} dx \leq \Unm.
\]
Hence,
\[
\int_B \frac{1}{\2ms + |x|^{\alpha}} |u(x)|^{\2ms + |x|^\alpha} dx \leq \int_B  | u(x) |^{\2ms + |x|^\alpha} dx \leq \Unm\tau^{\2ms } .
\]
Putting these facts together, we deduce that
\[
I(u) \geq \frac {\tau^2} 2  - \tau^{\2ms} \Unm
\]
for any $u \in \partial U$. Optimize the right hand side of the preceding inequality gives a suitable $\tau$ and a corresponding constant $\rho >0$. To realize the existence of $v$, we note that $v = R u_\varepsilon \in \Hmr$ for any $R\gg 1$. Moreover, because
\[
I(v) \leq  \frac {R^2} 2 \int_B |\nabla^m u_\varepsilon|^2 dx - R^{\2ms} \int_B \frac{1}{\2ms + |x|^{\alpha}} u_\varepsilon(x)^{\2ms + |x|^\alpha} dx,
\]
we know that $I(v)$ becomes negative if $R$ is large enough, thanks to $\2ms > 2$. Hence, we can choose any $R \gg 1$ such that $v \notin U$ and fix it.
\end{proof}

Let $R \gg 1$ and the corresponding $R u_\varepsilon = v \in \Hmr$ found in the proof of Lemma \ref{lemMPstructure}. Then, we define
\[
\Gamma:= \big\{\gamma :[0,R] \to \Hmr  \text{\rm is continuous},  \gamma(0) =0, \gamma(R) = R u_\varepsilon\big\}
\]
the set of continuous paths connecting $0$ and $v$ in $\Hmr$. Clearly, the set $\Gamma$ is not empty because the straight path  $\gamma_\varepsilon(t) = t u_\varepsilon$ with $t \in [0,R]$ belongs to $\Gamma$. Now we set
\begin{equation}\label{MPlevel}
c = \inf_{\ga \in \Gamma} \max_{u\in \ga} I(u).
\end{equation}
Because $I\big|_{\partial U} \geq \rho >0$, we deduce that $c \geq \rho >0$. The next lemma provides us an upper bound for $c$.

\begin{lemma}\label{lemUpperBound4Level}
The mountain-pass level $c$ of the functional $I$ satisfies
\[
c < \frac mn \Snm^{-n/m}.
\]
\end{lemma}

\begin{proof}

Hence, there holds
\begin{equation*}\label{eq:boundc}
c \leq \max_{t\in [0,R]} I(t u_\varepsilon) =: I (t_\varepsilon u_\varepsilon).
\end{equation*}
We first estimate the value of $t_\varepsilon$. Note that $t_\varepsilon \in (0,R)$, then $\frac{d}{dt}I(t u_\varepsilon)\big|_{t= t_\varepsilon} =0$, which implies that
\begin{equation}\label{eq:p1}
\int_B |\nabla^m u_\varepsilon|^2 dx = t_\varepsilon^{\2ms -2} \int_B t_\varepsilon^{|x|^\alpha} |u_\varepsilon|^{\2ms + |x|^\alpha} dx.
\end{equation}
Making use of Lemma \ref{Tiemcan} with $C=1$, we deduce that
\begin{equation}\label{eq:Tiemcan1}
\int_B |u_\varepsilon(x)|^{\2ms + |x|^\alpha} dx = \Snm^{-n/m} + \mathscr C_1 |\ln \varepsilon| \varepsilon^{\alpha} + o(\varepsilon^\alpha |\ln \vep|)_{\varepsilon \searrow 0}.
\end{equation}
Moreover, from \eqref{eq:tiemcan1}, we have
\begin{equation}\label{eq:tiemcan111}
\int_B |\nabla^m u_\varepsilon|^2 dx = \Snm^{-n/m} + O(\varepsilon^{n-2m})_{\varepsilon \searrow 0}.
\end{equation}
We claim that $t_\varepsilon \to 1$ as $\varepsilon \to 0$. Indeed, let us denote
\[
a_{\inf} = \liminf_{\varepsilon \searrow 0} t_\varepsilon \leq \limsup_{\varepsilon\to 0} t_\varepsilon = a_{\sup}
\]
and suppose that
\[
a_{\sup}>1.
\] 
Then there is some $\kappa >1$ and a subsequence $(t_{\varepsilon_i})$ such that $\varepsilon_i \to 0$ and $t_{\varepsilon_i}> \kappa$ for any $i$. This fact together with \eqref{eq:p1} and \eqref{eq:Tiemcan1} implies that
\[
\Snm^{-n/m} + O(\varepsilon_i^{n-2m})_{i \nearrow +\infty} \geq \kappa^{\2ms-2} \Big( \Snm^{-n/m} + \mathscr C_1 |\ln \varepsilon_i| \varepsilon_i^{\alpha} + o(\varepsilon_i^\alpha|\ln \vep_i|)_{i \nearrow +\infty} \Big) .
\]
Sending $i \nearrow +\infty$ to get a contradiction because $\kappa >1$ and $\2ms -2 > 0$. Hence, $a_{\sup}\leq 1$. By the same argument, we can also prove that
\[
a_{\inf} \geq 1.
\] 
This proves the claim, namely, $t_\varepsilon \to 1$ as $\varepsilon \to 0$. Consequently, we can choose $\varepsilon >0$ small enough such that $1/2 \leq t_\varepsilon \leq 3/2$. Moreover, we always have
\[
\2ms -2\leq \2ms+ |x|^\alpha -2 \leq \2ms -1
\]
for any $x \in B$. Consider the function $f$ defined on $[\2ms -2,\2ms -1] \times [1/2,3/2]$ by
\[
f(q,t) = \begin{cases}
\frac{t^q -1}{t-1} &\mbox{if $t\neq 1$,}\\
q&\mbox{if $t=1$.}
\end{cases}
\]
Obviously, the function $f$ is continuous and $f >0$ on $[\2ms -2,\2ms -1] \times [1/2,3/2]$. Hence 
\[
C_0:= \inf\{f(q,t)   :  (q,t) \in [\2ms -2,\2ms -1] \times [1/2,3/2]\} > 0.
\]
For $\varepsilon >0$ small enough, we have from \eqref{eq:p1}, \eqref{eq:Tiemcan1}, \eqref{eq:tiemcan111}, and $\alpha < \min\{n-2m, n/2\}$ that
\begin{align*}
C \varepsilon^{\alpha} (-\ln \varepsilon) + o(\varepsilon^\alpha|\ln \vep|)_{\varepsilon \searrow 0}&= \Big|\int_B |\nabla^m u_\varepsilon|^2 dx -\int_B |u_\varepsilon|^{\2ms + |x|^\alpha} dx\Big|\\
&=\Big|\int_B (t_\varepsilon^{\2ms-2+ |x|^\alpha} -1) |u_\varepsilon|^{\2ms + |x|^\alpha} dx\Big|\\
&=|t_\varepsilon -1| \int_B f(\2ms -2 + |x|^\al,t_\varepsilon) |u_\varepsilon|^{\2ms + |x|^\alpha} dx\\
&\geq C_0 |t_\varepsilon -1| \big[ \Snm^{-n/m} + \mathscr C_1 \varepsilon^{\alpha} (-\ln \varepsilon) + o(\varepsilon^\alpha|\ln \vep|)_{\varepsilon \searrow 0} \big] .
\end{align*}
Hence, $t_\varepsilon = 1 + R_\varepsilon$ with $R_\varepsilon =O(\varepsilon^\al(-\ln \varepsilon))_{\varepsilon \searrow 0}$. We have
\begin{align*}
\int_B \frac{1}{\2ms + |x|^{\alpha}} |t_\varepsilon u_\varepsilon(x)|^{\2ms + |x|^\alpha} dx 
&= \int_B \frac{(1+ R_\varepsilon)^{\2ms+ |x|^\alpha} -1}{\2ms + |x|^{\alpha}} |u_\varepsilon(x)|^{\2ms + |x|^\alpha} dx \\
&\quad + \int_B \frac{1}{\2ms + |x|^{\alpha}} |u_\varepsilon(x)|^{\2ms + |x|^\alpha} dx\\
&=I + II.
\end{align*} 
Using Taylor's expansion, we have 
\begin{align*}
(1 + R_\varepsilon)^{\2ms+ |x|^\alpha} &= 1 + (\2ms+ |x|^\alpha)R_\varepsilon \\
&\quad + (\2ms+ |x|^\alpha)(\2ms + |x|^\alpha -1)R_\varepsilon^2 \int_0^1 (1+ s R_\varepsilon)^{\2ms+ |x|^\alpha -2} (1-s) ds \\
&= 1 + (\2ms+ |x|^\alpha)R_\varepsilon + O(R_\varepsilon^2),
\end{align*}
and by \eqref{eq:Tiemcan1} we have
\begin{equation}\label{eq:I}
I = R_\varepsilon \int_B |u_\varepsilon(x)|^{\2ms + |x|^\alpha} dx + O(R_\varepsilon^2) = R_\varepsilon \Snm^{-n/m} + O(R_\varepsilon^2). 
\end{equation}
Putting \eqref{eq:I} and the estimate for II in \eqref{eq:II} gives
\begin{align*}
\int_B \frac{1}{\2ms + |x|^{\alpha}} &  |t_\varepsilon u_\varepsilon(x)|^{\2ms + |x|^\alpha} dx\\
 =& \frac1{\2ms} \Snm^{-n/m} +\frac {\mathscr C_1}{\2ms} |\ln \varepsilon| \varepsilon^{\alpha} + o(\varepsilon^\alpha|\ln \vep|) + R_\varepsilon \Snm^{-n/m} + O(R_\varepsilon^2). 
\end{align*}
Keep in mind that $R_\varepsilon = O(|\ln \varepsilon| \varepsilon^\alpha  )$ and that $\alpha \leq n-2m$. Therefore,
\begin{align*}
c \leq I(t_\varepsilon u_\varepsilon) &= \frac{(1+ R_\varepsilon)^2}2 \int_B |\nabla^m u_\varepsilon|^2 dx -\int_B \frac{1}{\2ms + |x|^{\alpha}} |t_\varepsilon u_\varepsilon(x)|^{\2ms + |x|^\alpha} dx\\
&=\frac12(1+ 2R_\varepsilon + R_\varepsilon^2)(\Snm^{-n/m} + O(\varepsilon^{n-2m})) -\frac1{\2ms} \Snm^{-n/m} - \frac {\mathscr C_1}{\2ms} |\ln \varepsilon| \varepsilon^{\alpha} \\
&\quad - R_\varepsilon \Snm^{-n/m} + o(\varepsilon^\alpha |\ln \vep|) + O(R_\varepsilon^2)\\
&= \frac mn \Snm^{-n/m}-\frac {\mathscr C_1}{\2ms} |\ln \varepsilon| \varepsilon^{\alpha}+ o(\varepsilon^{\alpha} |\ln \vep|).
\end{align*}
Taking $\varepsilon >0$ small enough we deduce that
\[
c < \frac mn \Snm^{-n/m}
\]
as claimed.
\end{proof}

In view of Lemmas \ref{lemMPstructure} and \ref{lemUpperBound4Level} above, there is a (PS) sequence $(u_j)_j$ in $\Hmr$ such that
\[
I(u_j) \to c \in \big(0,  \frac mn \Snm^{-n/m}\big)
\]
and
\[
I'(u_j) =o(1)_{j \nearrow +\infty}.
\] 

\begin{lemma}\label{lemPSisBounded}
The sequence $(u_j)$ is bounded in $\Hmr$.
\end{lemma}

\begin{proof}
The argument is standard. Indeed, because $I'(u_j) \to 0$ we have
\[
\int_B (u_j)_+^{\2ms + |x|^\alpha} dx = \int_B |\nabla^m u_j|^2 dx + o(1)_{j \nearrow +\infty} \|\nabla^m u_j\|_{L^2(B)}.
\]
Because $I(u_j) \to c$ we obtain
\begin{align*}
\int_B |\nabla^m u_j|^2 dx &= 2c + \int_B \frac{2}{\2ms + r^\alpha} |u_j|^{\2ms + |x|^\alpha} dx + o(1)_{j \nearrow +\infty} \\
& \leq 2c+ \frac{2}{\2ms} \int_B (u_j)_+^{\2ms + |x|^\alpha} dx + o(1)_{j \nearrow +\infty}.
\end{align*}
Combining these two estimates, we arrive at
\[
\int_B |\nabla^m u_j|^2 dx \leq \2ms c + o(1)_{j \nearrow +\infty} \|\nabla^m u_j\|_{L^2(B)},
\]
which implies that the sequence $(u_j)$ is bounded in $\Hmr$. 
\end{proof}

In view of Lemma \ref{lemPSisBounded}, up to a subsequence, still denoted by $(u_j)$, there is some $u \in \Hmr$ such that \begin{itemize}
  \item $u_j \rightharpoonup u$ weakly in $\Hmr$, 
  \item $u_j \to u$ strongly in $H_{0,r}^k(B)$ for any $0\leq k < m$, 
  \item $u_j \to u$ a.e. in $B$, and 
  \item $\|\nabla^m u_j\|_{L^2(B)} \to l\geq 0$,
\end{itemize}
as $j \to +\infty$. Hence, we have
\begin{equation}\label{eq:1}
\int_{B} (u_j)_+^{\2ms + |x|^\alpha} dx = l^2 + o(1)_{j \nearrow +\infty} 
\end{equation}
and
\begin{equation}\label{eq:12}
\int_B \frac{(u_j)_+^{\2ms + |x|^\alpha}}{\2ms + |x|^\alpha} dx =\frac{l^2}2 -c + o(1)_{j \nearrow +\infty}.
\end{equation}
Evidently, $l >0$ because if otherwise \eqref{eq:12} gives a contradiction if we let $j$ large enough because $c >0$. 

We now rule out the possibility that $u \equiv 0$. 

\begin{lemma}\label{nontriviallimit}
The weak limit $u$ is non-trivial, namely $u \equiv 0$.
\end{lemma}

\begin{proof}
Note that $\|\nabla^m u_j\|_{L^2(B)} \to l$. By \eqref{eq:HR}, we have
\[
\int_0^1 |u_j'(r)|^2 r^{n-2m +1} dr \leq C
\]
for some constant $C >0$ independent of $j$. Then, via a density argument, the estimate \eqref{eq:est1} implies that
\begin{equation}\label{eq:est1j}
|u_j(x)| \leq \Big( \frac C{n-2m+2} \Big) ^{1/2} |x|^{-\frac{n-2m+2}2}
\end{equation}
for a.e. $x$. Since $u_j\to 0$ a.e. in $B$, by Lebesgue's dominated convergence theorem, we have
\begin{equation}\label{eq:truna}
\begin{aligned}
\lim_{j \nearrow +\infty} \int_{B\setminus B_a} |u_j|^{\2ms} dx &= \lim_{j \nearrow +\infty} \int_{B\setminus B_a} |u_j|^{\2ms+|x|^\alpha} dx \\
&=\lim_{j \nearrow +\infty} \int_{B\setminus B_a}\frac{ |u_j|^{\2ms+|x|^\alpha}}{\2ms + |x|^\alpha} dx =0
\end{aligned}
\end{equation}
for any $a>0$. Let $a_0 >0$ be such that $ \big( \frac C{n-2m+2} \big) ^{1/2} a_0^{-\frac{n-2m+2}2} \geq 1$. Hence, by \eqref{eq:est1j} and \eqref{eq:truna}, for any $a < a_0$ we have
\begin{align*}
\int_B (u_j)_+^{\2ms + |x|^\alpha} dx &= \int_{B_a} (u_j)_+^{\2ms + |x|^\alpha} dx + o(1)_{j \nearrow +\infty}\\
&\leq \Big( \big( \frac C{n-2m+2} \big) ^{1/2} a^{-\frac{n-2m+2}2} \Big) ^{a^\alpha} \int_B |u_j|^{\2ms} dx + o(1)_{j \nearrow +\infty}\\
&\leq \Big( \big( \frac C{n-2m+2} \big) ^{1/2} a^{-\frac{n-2m+2}2} \Big) ^{a^\alpha} \Smn \|\nabla^m u_j\|_{L^2(B)}^{ \2ms} + o(1)_{j \nearrow +\infty},
\end{align*}
here we have used the Sobolev inequality once. Thanks to $\alpha > 0$, letting $a\to 0$ gives
\[
\int_B (u_j)_+^{\2ms + |x|^\alpha} dx \leq \Smn \|\nabla^m u_j\|_{L^2(B)}^{ \2ms} + o(1)_{j \nearrow +\infty},
\]
which, after letting $j \to +\infty$ and making use of \eqref{eq:1}, yields
\begin{equation}\label{eq:2}
l \geq \Snm^{-n/(2m)}.
\end{equation}
With a help from \eqref{eq:truna}, we have
\begin{align*}
\Big|\int_B \frac{(u_j)_+^{\2ms + |x|^\alpha}}{\2ms + |x|^\alpha} dx &-\int_B \frac{(u_j)_+^{\2ms + |x|^\alpha}}{\2ms} dx\Big| \\
&=\int_B \frac{|x|^\alpha}{\2ms(\2ms + |x|^\alpha)} (u_j)_+^{\2ms + |x|^\alpha} dx\\
&\leq \frac{a^\alpha}{(\2ms)^2} \int_{B_a} |u_j|^{\2ms + |x|^\alpha} dx + \frac1{(\2ms)^2} \int_{B\setminus B_a} |u_j|^{\2ms + |x|^\alpha} dx \\
&\leq \frac{a^\alpha}{(\2ms)^2} \int_{B} |u_j|^{\2ms + |x|^\alpha} dx + o(1)_{j \nearrow +\infty}.
\end{align*}
Letting $a\searrow 0$ and making use of \eqref{eq:12} to get
\[
 \frac 1{\2ms} \int_B (u_j)_+^{\2ms + |x|^\alpha} dx =\int_B \frac{(u_j)_+^{\2ms + |x|^\alpha}}{\2ms + |x|^\alpha} dx + o(1)_{j \nearrow +\infty} = \frac{l^2}{2} -c + o(1)_{j \nearrow +\infty}.
\] 
Now letting $j \nearrow +\infty$ and using \eqref{eq:1}, we get
\[
l^2 = \frac nm c,
\]
which, by Lemma \ref{lemUpperBound4Level}, gives
\[
l^2< \Snm^{-n/m},
\]
which contradicts  \eqref{eq:2}. Hence, we have just shown that $u\not\equiv 0$. 
\end{proof}

\begin{lemma}\label{lemExistenceWeak}
The weak limit $u$ solves \eqref{eq:supercriticaleq}.
\end{lemma}

\begin{proof}
Since $u_j\rightharpoonup u$ weakly in $\Hmr$ and $I'(u_j) \to 0$, we conclude that $u$ is a non-trivial, weak solution to 
\[
\left\{
\begin{aligned}
(-\De)^m u  &=  u_+^{\2ms + |x|^\alpha -1}& \mbox{ in } & B,\\
\pa_r^j u &= 0&\mbox{ on }& \partial B , \quad   j= 0,1,\ldots,m-1.
\end{aligned}
\right.
\]
Making use of \cite[Theorem 5.1]{GGS} we deduce that $u\geq 0$ in $B$. By Lemma \ref{nontriviallimit}, we know that $u_+^{\2ms + |x|^\alpha -1} \geq 0$ and is not identical $0$ in $B$. Again, by \cite[Theorem 5.1]{GGS}, we further get $u >0$ in $B$. Therefore, $u$ is indeed a weak solution to \eqref{eq:supercriticaleq}.
\end{proof}

\section*{Acknowledgments}

The research of Q.A.N is funded by the Vietnam National University, Hanoi (VNU) under project number QG.19.12. The research of V.H.N is partially funded by the Simons Foundation Grant Targeted for Institute of Mathematics, Vietnam Academy of Science and Technology. 


\addtocontents{toc}{\protect\setcounter{tocdepth}{0}}

\section*{ORCID iDs}

\noindent Qu\cfac oc Anh Ng\^o: 0000-0002-3550-9689

\noindent Van Hoang Nguyen: 0000-0002-0030-5811

\addtocontents{toc}{\protect\setcounter{tocdepth}{2}}






\end{document}